\documentclass{amsart}

\usepackage{amsmath, amsthm, amsfonts}
\usepackage{wrapfig}
\usepackage{amssymb}
\usepackage{bbm}
\usepackage{graphicx}
\usepackage{float}

\newtheorem{theorem}{Theorem}
\newtheorem{lemma}[theorem]{Lemma}
\newtheorem{prop}[theorem]{Proposition}
\newtheorem{defn}[theorem]{Definition}
\newtheorem{cor}[theorem]{Corollary}

\newtheorem*{remark}{Remark}

\renewenvironment{remark}[1][Remark]{\begin{trivlist}
\item[\hskip \labelsep {\bfseries #1}]}{\end{trivlist}}

\newcommand{\lt}{\left}
\newcommand{\rt}{\right}
\newcommand{\bpm}{\begin{pmatrix}}
\newcommand{\epm}{\end{pmatrix}}
\newcommand{\bsm}{\lt(\begin{smallmatrix}}
\newcommand{\esm}{\end{smallmatrix}\rt)}

\renewcommand{\d}{\mathrm{d}}
\newcommand{\ds}{\mathrm{d}s}
\newcommand{\dw}{\mathrm{d}w}
\newcommand{\dv}{\mathrm{d}v}
\newcommand{\dt}{\mathrm{d}t}
\newcommand{\du}{\mathrm{d}u}
\newcommand{\dx}{\mathrm{d}x}
\newcommand{\dr}{\mathrm{d}r}
\newcommand{\dth}{\mathrm{d}\theta}

\newcommand{\dy}{\frac{\mathrm{d}y}{y}}
\newcommand{\dz}{\frac{\mathrm{d}x \mathrm{d}y}{y^2}}
\newcommand{\dxi}{\mathrm{d}\xi}

\newcommand{\D}{\ensuremath{\mathbb{D}}}
\newcommand{\Z}{\ensuremath{\mathbb{Z}}}
\newcommand{\N}{\ensuremath{\mathbb{N}}}

\newcommand{\R}{\ensuremath{\mathbb{R}}}
\newcommand{\C}{\ensuremath{\mathbb{C}}}

\renewcommand{\H}{\ensuremath{\mathbb{H}}}

\newcommand{\vep}{\varepsilon}

\DeclareMathOperator{\GL}{GL}
\DeclareMathOperator{\SL}{SL}

\renewcommand{\Im}{\operatorname{Im}}
\renewcommand{\Re}{\operatorname{Re}}
\DeclareMathOperator*{\Res}{Res}

\newcommand{\aquad}{\qquad\qquad}
\newcommand{\bquad}{\aquad\aquad}
\newcommand{\cquad}{\bquad\bquad}

\renewcommand{\th}{\textsuperscript{th}}
\newcommand{\inv}{^{-1}}
\newcommand{\hf}{\frac 12}
\newcommand{\qtr}{\frac14}
\newcommand{\Vol}{\operatorname{Vol}}

\renewcommand{\mod}{\text{ mod }}

\usepackage{tensor}
\usepackage{bbm}
\newcommand{\kron}[2]{\lt(\frac{#1}{#2}\rt)}

\title{Subconvexity for Half Integral Weight $L$-functions}
\author{Eren Mehmet K\i ral}
\date{\today}

\begin{document}

\begin{abstract}
	We prove a subconvexity bound in the conductor aspect for the $L$-function $L(s,f,\chi)$ where $f$ is a half integer weight modular form. This $L$-function has analytic continuation and functional equation, but no Euler product. Due to the lack of an Euler product, one does not expect a Riemann hypothesis for half integer weight modular forms. Nevertheless one may speculate a Lindel\"{o}f-type hypothesis, and this current subconvexity result is an indication towards its truth.

\end{abstract}

\maketitle


\section{Introduction}\label{sec:Introduction}
	A very general definition of $L$-functions and their functional equations, given in chapter 5 of \cite{IwaniecAnalyticBook}, goes roughly as follows. Let $f$ be some function of arithmetic interest (or a meaningless symbol) with complex numbers $A_f(n)$ for each $n \in \N$. Assume that the associated $L$-function
	\[
		L(s,f) = \sum_{n=1}^\infty \frac{A_f(n)}{n^s}
	\]
	is convergent for $\Re(s) >1$. Let
	\[
		\gamma(s,f) = \pi^{-ds/2} \prod_{j=1}^d \Gamma\lt(\frac{s+\kappa_j}{2}\rt)
	\]
	be a product of gamma factors.  Here the $\kappa_j$ are either real or come in complex conjugate pairs, and they satisfy $\Re(\kappa_j) >-1$. The integer $d$ is called the degree of $L(s,f)$. Let $q(f)\geq 1$ be an integer called the conductor of $L(s,f)$. The function
	\[
		L^* (s,f) := q(f)^{\frac s2} \gamma(s,f) L(s,f),
	\]
	is called the completed $L$-function associated to $L(s,f)$, and often satisfies a functional equation of the form
	\[
		L^*(s,f) = \vep(f)L^*(1-s,\tilde{f}).
	\]
	Here $\vep(f)$ is a complex number with $|\vep(f)|=1$, called the root number of $L(s,f)$, and $\tilde{f}$ is called the dual of $f$, satisfying $A_{\tilde{f}}(n) = \overline{A_f(n)}$. The completed $L$-function $L^*(s,f)$ is a meromorphic function of order $1$ with at most poles at $s = 0$ and $s=1$. 
	
	The above conditions encompass a wide range of important functions in analytic number theory.  For any such $L$-function, the functional equation reveals the value of the function in the left half-plane $\Re(s)<0$ as a reflection of the corresponding value in the right half-plane $\Re(s)>1$. The region in between is called the critical strip, with midpoint $1/2$. 
	
	An important question in many contexts is the size of the value of the $L(s,f)$ at the center of the critical strip, i.e.\ at $s = \tfrac12 + it$. The Phragm\'{e}n-Lindel\"{o}f convexity principle applied to the points $s = 1 +\vep +it$ and $s = -\vep +it$ implies,
	\begin{equation}\label{eq:convexity}
		L\left(\tfrac12+it, f\rt) \ll \lt(q(f)(1+|t|)^d\prod_{j=1}^d \lt(1+|\kappa_j|\rt)\rt)^{\qtr + \vep}.
	\end{equation}
	 This is called a convexity bound. 
	 
	 As an example let $f$ be a holomorphic modular form of even weight $k$ and level $N$. Consider the $L$ function of $f$ twisted by a Dirichlet character $\chi$ of conductor $Q$. Then the conductor of $L(s,f\otimes\chi)$ is $q(f) = NQ^2$, and the $\kappa$'s are $(k \pm 1)/2$. Any improvement on \eqref{eq:convexity} in any of the aspects, imaginary part, level, conductor of twisting character, or weight is called a subconvexity bound in that particular aspect.

	 There is a wide class of functions that satisfy all the properties above, and furthermore have an Euler product. 
	 \[
	 	L(s,f) = \prod_{p \text{ prime}} \prod_{j=1}^d (1-\alpha_j(p) p^{-s})\inv,
	 \]
	 for some complex numbers $\alpha_j(p)$ of absolute value 1. It is generally believed that these functions satisfy a Riemann hypothesis, which is to say that the completed $L$-functions only vanish on the critical line $\Re(s) = 1/2$. Conrey and Ghosh have proven in \cite{conrey2006remarks} that any sich $L$-function as above, satisfying the Riemann hypothesis must also satisfy the generalized Lindel\"{o}f hypothesis. That is, 
	 \[
	 	L(\tfrac 12 + it, f) \ll_{\vep, d} \lt(q(f)(1+|t|)^{d} \prod_{j=1}^d (1+|\kappa_j|) \rt)^\vep.
	 \]
	 Therefore any subconvexity bound should be considered as an improvement towards the Lindel\"{o}f hypothesis.
	 	 
	 Many instances of subconvexity bounds have been proven. For example Martin Huxley in \cite{Huxley} has proven that
	 \[
	 	\zeta(\tfrac 12 + it) \ll t^{\frac{32}{205} + \vep}
	 \]
	 for any epsilon. For the zeta function the analytic conductor only has the $t$ aspect. In the case of a Dirichlet $L$-function $L(s,\chi)$, for a Dirichlet character $\chi$ modulo $Q$, the $L$-function has conductor $Q$. An example of a subconvexity bound in the conductor aspect is
	 \[
	 L(\tfrac12, \chi) \ll Q^{\frac 3{16} + \vep}.
	 \]
	 as was proven by Burgess in \cite{burgess1963character}.
	 
	 In these instances the $L$-functions in question do have Euler products with the given restriction on the magnitude of $\alpha_j(p)$ and furthermore, they are suspected of satisfying a Riemann hypothesis hence a Lindel\"of hypothesis. Conversely it is possible to find examples of Dirichlet series satisfying all the properties above, except for the Euler product, which do not satisfy any bound better than the convexity bound. A very simple example of such a Dirichlet series was given in \cite{conrey2006remarks}. Let
	\begin{equation*}
			D(s) = \sum_{n=1}^\infty \frac{d(n)\cos\lt(\tfrac{2\pi n}{q}\rt)}{n^s}.
	\end{equation*}
	Here $q$ is a prime and $d(\cdot)$ is the divisor function. In fact this Dirichlet series has conductor $q^2$ and $D(1/2)$ gets as large as $\sqrt{q}\log q$ as $q$ goes to infinity through primes. This is the convexity bound, and happens to be the best possible.
	
	This counterexample may seem to suggest that the existence of an Euler is crucial for a Lindel\"of hypothesis to hold.  Nevertheless there may be other conditions under which a Lindel\"of type hypothesis might be true. What causes an Euler product in the case of $L$-functions of automorphic forms is the fact that the automorphic forms are simultaneous eigenfunctions for the Hecke operators. Therefore one might suspect that it is the property of being a simultaneous eigenfunction of all the Hecke operators that is the essential property in implying the Lindel\"of Hypothesis. This was conjectured informally by Jeffrey Hofftstein in Obefwohlfach in 2011. The object of this paper is to provide some evidence that this may be the case.

	One interesting class of $L$-functions is the collection of Mellin transforms of newforms on the $n$-fold cover of $\GL_2$, with $n \geq 2$. These can be simultaneously diagonalized by the Hecke operators, but, except for the case of the quadratic theta function, these $L$-functions do not have Euler products.  More precisely, let $k$ be a half-integer, and $f$ a holomorphic cusp form of weight $k$ and level $N$. Let $\chi$ be a primitive Dirichlet character of modulus $Q$. Assume that $f$ has the Fourier expansion,
		\[
			f(z) = \sum_{n=1}^\infty A(n) n^{\frac {k-1}{2}} e(nz).
		\]
		For $s$ with $\Re(s)>1$, define the $L$-series
		\[
			L(s,f,\chi) := \sum_{n=1}^\infty \frac{A(n)\chi(n)}{n^s}.
		\]
		This function has analytic continuation and a functional equation to all $s$. The main objective of this paper is to prove the following theorem, which shows that 
despite the lack of an Euler product, there is more structure for the $L$-function of a half integer weight modular form, than is implied by the convexity principle.

	\begin{theorem}\label{thm:main}
		Let $f$ be a half integer modular form, $\chi$ a primitive Dirichlet character modulo $Q$. For any $\vep >0$,
		\[
			L(\tfrac 12,f,\chi) \ll_{f,\vep} Q^{\frac 38+\frac{\theta}{4}+\vep}.
		\]
		Here $\theta$ is the best progress towards the Ramanujan conjecture for the Fourier coefficients of level zero Maass forms.
	\end{theorem}
	\begin{remark}
	The number $\frac38$ we get is what is called a Burgess type bound in the literature of subconvexity bounds. This is a name given to a bound which is three quarters of the convexity bound in any of the aspects of the analytic conductor. Such a bound appears frequently, even if achieved through different means. It seems to be a natural boundary for several methods, and the name is given in honor of the bound in \cite{burgess1963character}. Plugging in the Kim-Sarnak bound $7/64$ for the value of $\theta$ gives us $L(\tfrac12,f,\chi) \ll Q^{0.40234\ldots}$.
	\end{remark}
	
	In a similar vein, a subconvexity bound for a double Dirichlet series 
	\begin{equation}\label{eq:ValentinDoubleDirichlet}
		Z(s,w;\psi,\psi') = \zeta(2s+2w-1)\sum_{d} \frac{L_2(s,\chi_d \psi)\psi'(d)}{d^w}
	\end{equation}
	for the $\Im(s)$ and $\Im(w)$ aspects was obtained by Valentin Blomer in \cite{blomer2011subconvexity}. Specializing to $s = 1/2$, this gives a subconvexity bound for a Dirichlet series with a functional equation but without an Euler product. This, and our $L$-function of a half integer weight modular form can be considered to be of similar origin as one can identify \eqref{eq:ValentinDoubleDirichlet} with the Mellin transform of a half integral weight Eisenstein series. Our Theorem \ref{thm:main} shows that twists of Mellin transforms of half integer weight cusp form satisfy a subconvexity bound in the conductor aspect.
	
	\subsection{A sketch of the approach}
	One of the critical issues that arise when approaching a problem like this from a spectral point of view is the bounding from above of a certain triple product.   This was first accomplished by  Sarnak in \cite{SarnakIMRN94}, where he uses a geometric method to prove exponential decay in the eigenvalue of triple inner products of weight zero Maass forms.  A corresponding bounding of a triple inner product in the case of integral weight cusp forms is done in the appendix of \cite{JeffShiftedSum} by Andrei Reznikov using the uniqueness of trilinear forms in certain automorphic representations.  The main contribution of this work is the proof of  a similar result when two of the automorphic forms in the triple product are half integer weight modular forms.  In this case the representation theoretic method of Reznikov is not applicable.  This is done in  sections \ref{sec:AutomorphicKernel} to \ref{sec:selbergTransform}.  We roughly follow the method of Sarnak, obtaining explicit bounds on the polynomial growth attached to the exponential decay, depending on the weight and the sup-norms of the modular forms.

		With this vital ingredient we proceed as follows, along the lines of \cite{duke2002subconvexity}. The notation that is used is given in section \ref{sec:notation}. We begin, in section \ref{sec:amplification}, by showing that it is enough to bound expressions of the form 
	\[
		\lt|\sum_m A(m) \chi(m) H\lt(\frac mx\rt)\rt|,
	\]
	for a smooth and compactly supported function $H$.	We then consider  an amplified sum,
	\begin{equation}\label{eq:amplifiedsum}
		\sum_{\psi \pmod Q} \lt|\sum_m A(m) \chi(m) H\lt(\frac mx\rt)\rt| \lt|\sum_{\substack{L \leq \ell\leq 2L \\ \ell \text{ prime}}}\psi(\ell)\overline{\chi(\ell)}\rt|^2
	\end{equation}
	After expanding the squares in both terms, and applying cancellation due to the orthogonality of characters  two types of terms appear in this summation: those we call \emph{diagonal} terms and those called \emph{off-diagonal} terms. 
	
	 Bounding the diagonal term is easier and is done in section \ref{sec:diagonalSum} with the aid of the Rankin-Selberg type Dirichlet series,
	 \[
	 	\sum_{m\geq1} \frac{a(m)\overline{b(m)}}{m^s}
	 \]
	 where $a(m)$ and $b(m)$ are Fourier coefficients of the oldforms $f(\ell_1z), f(\ell_2z)$ respectively. This Dirichlet series is obtained by convolving $f(\ell_1z)\overline{f(\ell_2z)}y^k$ with an Eisenstein series.
	   
	   In section \ref{sec:doubleDirichlet} we estimate the off-diagonal terms.  To that end we consider a shifted convolution double Dirichlet series
	\[
		Z(s,w) = \sum_{m\geq1, h\geq1} \frac{a(m)\overline{b(m+h)}}{m^s h^{w+k-1}}.
	\]
	The off diagonal terms are interpreted as an inverse Mellin transform of $Z_Q$,
	and in section \ref{sec:movingLines} we bound these terms by manipulating the inverse Mellin transform of the function $Z_Q$.
	
	We conclude the proof of Theorem \ref{thm:main} in section \ref{sec:subconvexity}.
	
	The exposition sketches the parts of proof that differ only very slightly from \cite{JeffShiftedSum}. Full details are given for Sections \ref{sec:diagonalSum} and \ref{sec:AutomorphicKernel}-\ref{sec:subconvexity}.


\section{Notation}\label{sec:notation}
	
	For any complex number $z \in \C$, set
	\[
		e(z) = e^{2\pi i z}.
	\]
	
	Let $\H = \{z = x + iy \in \C: y>0\}$, the upper half plane.
	
	Let $\Gamma_0(N)$ be the set of integral $2\times2$ matrices such that the lower left entry is divisible by $N$. We shall also denote $\Gamma =\Gamma_0(N)$. The set of inequivalent cusps of $\Gamma$ will be given by
	\[
	\lt\{u/w: w|N, \,\, (u,w) = 1,\,\, u\!\! \mod (w,N/w)\rt\}.
	\]
	We will make the assumption that $N/4$ will be squarefree and odd. Therefore the cusps $\mathfrak{a}$ will be of the form $1/w$ where $w|N$. 
	
	Let us denote the volume of the modular surface $\Gamma_0(N) \backslash \H$ by $\mathcal{V}$. The exact formula for the volume is given by $$\mathcal{V} = \frac \pi3 [\SL(2,\Z):\Gamma_0(N)] = \frac\pi3 N \prod_{p|N}\lt(1+\frac 1p\rt).$$ 
	
	Given $\gamma = \bsm a&b\\c&d\esm \in \Gamma_0(N)$, define the half integral weight cocyle as
	\[
		j(\gamma,z) = \vep_d\inv \kron{c}{d} (cz+d)^\hf.
	\]
	Here $\kron{c}{d}$ denotes the Kronecker symbol as in \cite{ShimuraHalfInt73}. Furthermore
	\[
		\vep_d = \begin{cases}1 &\text{ if } d \equiv 1 \mod 4\\ i &\text{ if } d \equiv 3 \mod 4 \end{cases}.
	\]
	
	Let $\chi$ be a Dirichlet character of modulus dividing $N$. Let $f$ be a holomorphic cusp form of level $N, (4|N)$, nebentypus $\chi$ and weight $k$, a half-integer. This means that for any $\gamma = \bsm a&b\\c&d  \esm\in \Gamma_0(N)$, 
	\[
		f(\gamma z) = \chi(d) j(\gamma,z)^{2k} f(z)
	\]
	is satisfied. Furthermore the constant term of the Fourier expansions at all cusps vanish. More details are given in \cite{ShimuraHalfInt73}. If the nebentypus of a modular form is not mentioned then it is assumed to have a trivial character.
	
	Furthermore, assume that $f$ is an eigenfunction of the operator 
	\[
		(\tilde{W}_Nf)(z) = \lt. f \rt|_{k,\lt(\begin{smallmatrix} 0&-1\\N&0\end{smallmatrix}\rt)}(z) = i^{-k}N^{\frac k2} (Nz)^{-k} f\lt(\frac{-1}{Nz}\rt).
	\] 
	The eigenvalue has to be $\pm 1$. Let us call this eigenvalue $\vep(f)$. Set 
	\[
		f(z) = \sum_{n=1}^\infty a(n) e(nz) = \sum_{n=1}^\infty A(n) n^{\frac{k-1}{2}}e(nz).
	\] 
	The Fourier coefficients are thus normalized as this lets the functional equation be symmetric around $s = \frac12$. For $\chi$ a Dirichlet character modulo $Q$, let the $L$-function of $f\otimes \chi$ be defined as
	\begin{equation}\label{eq:Lfunction}
		L(s,f,\chi) = \sum_{n=1}^\infty \frac{A(n)\chi(n)}{n^s}. 
	\end{equation}
	Define the completed $L$-function of $f_\chi$ to be,
	\begin{align*}
		L^*(s,f,\chi) &:= (\sqrt{N}Q)^{s} \int_0^\infty f_\chi(iy) y^{s+\frac{k-1}{2}} \dy\\
		& =(\sqrt{N}Q)^s (2\pi)^{-(s+\frac{k-1}{2})} \Gamma\lt(s+\tfrac{k-1}{2}\rt) \sum_{n=0}^\infty \frac{A(n)\chi(n)}{n^s}\\
 		&= 2\pi^{\frac{k-1}{2}} \lt(NQ^2\rt)^{\frac{s}{2}} \pi^{-s} \Gamma\lt(\frac{s + \frac{k-1}{2}}{2}\rt)\lt(\frac{s + \frac{k+1}{2}}{2}\rt)L(s,f,\chi).
	\end{align*}
	Here $f_\chi(z) = \sum_{n=1}^\infty A(n)n^{\frac{k-1}{2}}\chi(n) e(nz)$ is the twisted modular form of level $NQ^2$. We can realize $f_\chi$ as an average over the additive twists:
	\[
		f_\chi(z) = \frac{1}{g(1,\overline{\chi})} \sum_{u \pmod Q} \overline{\chi(u)} f\lt(z + \frac uQ\rt).
 	\]
 	Using this we obtain
 	\[
 		L^*(s,f,\chi) = \frac{1}{g(1,\overline{\chi})} \sum_{u \pmod Q} \overline{\chi(u)} L^*(s,f, \tfrac{u}{Q}),
	\]
		where 
	\[
		L(s,f,\tfrac uQ) := (\sqrt{N}Q)^s \int_0^\infty f\lt(iy + \frac uQ\rt) y^{s + \frac{k-1}{2}} \dy.
	\]
	Also here $g(n,\psi)$ stands for the Gauss sum for the integer $n$ and the character $\psi$ of modulus $Q$, not necessarily primitive:
	\[
		g(n,\psi) := \sum_{u \pmod Q} \psi(u) e\lt(\frac{nu}{Q}\rt).
	\]
	
	We sketch the proof of the functional equation of $L^*(s,f,\chi)$ in case $(Q,N) = 1$. Using the matrix identity
	\[
		\lt(\begin{matrix} 1&u/Q\\ 0&1 \end{matrix}\rt) \lt(\begin{matrix} 0&-1/Q\\ NQ&0\end{matrix}\rt)
		=
		\lt(\begin{matrix} (Nuv + 1)/Q & u\\ Nv& Q \end{matrix}\rt)\lt(\begin{matrix}0&-1\\N&0\end{matrix}\rt)
		\lt(\begin{matrix} 1 & v/Q\\ 0&1\end{matrix}\rt),
	\]
	we can show that
	\[
		L^*(s,f,\tfrac{u}{Q}) = \vep(f) \vep_Q^{-2k} \kron{Nv}{Q} L^*(1-s, f, \tfrac vQ)
	\]
	where $v$ is an integer which satisfies $Nuv \equiv -1 \pmod Q$.
	
	This implies a functional equation 
	\begin{equation} \label{eq:multiplicativelyTwistedFunctionalEquation}
		L^*(s,f,\chi) = \vep^*(f,\chi) L^*(1-s,f,\chi,\chi')
	\end{equation}
	where $\vep^*(f,\chi)$ is a quantity of absolute value $1$, explicitly given as,
	\[
		\vep^*(f,\chi) := \vep(f) \vep_Q^{-2k} \kron{N}{Q} \chi(-N).
	\]
	We mean, by the $L$-function on the right,
	\[
		L^*(s,f,\chi,\chi') = \lt(\sqrt{N}Q\rt)^s (2\pi)^{-(s + \frac{k-1}{2})} \Gamma\lt(s + \tfrac {k-1}{2}\rt) \frac{1}{g(1,\overline{\chi})} \sum_{n=1}^\infty \frac{A(n) g(n,\chi')}{n^s}
	\]
	with 
	\[
		\chi' = \chi \kron{\cdot}{ Q}.
	\]
	
	Notice that if $\chi' = \chi \kron{\cdot}{Q}$ is primitive modulo $Q$, then  \eqref{eq:multiplicativelyTwistedFunctionalEquation} implies $L^*(s,f,\chi) = \vep(f,\chi) L^*(1-s,f,\chi') $ where $\vep(f,\chi) = \vep^*(f,\chi) g(1, \chi') g(1,\overline{\chi})^{-1}$.
	
	The functional equation \eqref{eq:multiplicativelyTwistedFunctionalEquation} also implies the  approximate functional equation as in \cite{IwaniecAnalyticBook}. 
	
	The approximate functional equation of $L(s,f,\chi)$ is given as follows. For a function $V$ on the positive real line, satisfying 
	\begin{align*}
	x^jV_s^{(j)}(x) &\ll \lt(1+ \frac{x}{k}\rt)^{-A}\\
	x^jV_s^{(j)}(x) &= \delta_{j,0} + O\lt(\lt(\frac{x}{k}\rt)^\alpha\rt)
	\end{align*}
	with $A>0, 0<\alpha< \tfrac 13 (\tfrac{k}{2} + \Re(s))$, we have the equality 
	\begin{align}\label{eq:approximateFunctional}
		L(\tfrac12,f,\chi) =& \sum_{m=1}^\infty \frac{A(m)\chi(m)}{\sqrt{m}}V\lt(\frac{2\pi m}{\sqrt{N}Q}\rt) \\
		&+  \vep^*(f,\chi)\frac{1}{g(1,\overline{\chi})} \sum_{m=1}^\infty \frac{\overline{A(m)}g(m,\chi')}{\sqrt{m}} V\lt(\frac{2\pi m}{\sqrt{N}Q}\rt).\notag
	\end{align} 
	
	Let $L^2(\Gamma_0(N)\backslash \H)$ be the Hilbert space of square integrable functions on the upper half plane, invariant under $\Gamma_0(N)$. The inner product, making this space a Hilbert space is given by
	\[
		\langle \phi_1,\phi_2 \rangle = \frac{1}{\mathcal{V}} \iint_{\Gamma_0(N)\backslash \H}\phi_1(z)\overline{\phi_2(z)}\dz.
	\]
	This is a Riemann surface, we call the differential operator
	\[
		\Delta = -y^2\lt(\frac{\partial^2}{\partial x^2} + \frac{\partial^2}{\partial y^2}\rt)
	\]
	the Laplace-Beltrami operator, and nonconstant eigenvalues of this operator are called Maass forms. 
	
	We write the Fourier expansion of a Maass form $u_j$ with Laplace eigenvalue $\lambda_j = \tfrac{1}{4} + t_j^2$, normalized to have $L^2$ norm equal to one, as,
	\begin{equation}\label{eq:FourierMaass}
		u_j(z) = \sum_{m\neq 0} \rho_j(m) \sqrt{y}K_{it_j}(4\pi |m| y)e(mx).
	\end{equation}
	Further the Eisenstein series at the cusp $\mathfrak{a}$ has a Fourier expansion which is given by the following proposition, taken from \cite{blomer2004shifted}
		\begin{prop}\label{prop:eisensteinFourier}
			Let $E_{\mathfrak{a}}^{(N)}(z,s)$ be the Eisenstein series at the cusp $\mathfrak{a}$ of $\Gamma_0(N)$. Note that with our assumption that $N/4$ is odd and squarefree all cusps $\mathfrak{a} = 1/w$ for $w |N$. Then the Fourier expansion of is given by,
			\begin{align*}
				E_{\mathfrak{a}}(z,s)  =& \delta_{\mathfrak{a},\infty} y^s + \frac{\zeta^*(2s-1)}{\zeta^*(2s)}\rho_{\mathfrak{a}}(s) y^{1-s} \\
				&+ \frac{2\pi^s \sqrt{y}}{\Gamma(s)} \sum_{m\neq 0} |m|^{s-\hf}\rho_{\mathfrak{a}}(s,m)K_{s-\hf}(2\pi |m|y)e(mx),
			\end{align*}
			where $\zeta^*(s) = \zeta(s)\Gamma\lt(\tfrac s2\rt) \pi^{-\frac s2}$ is the completed zeta function,
			\[
				\rho_{\mathfrak{a}}(s) = \phi(w)\lt(\frac{1}{wN}\rt)^s \prod_{p|N}(1-p^{-2s})\inv \prod_{p|\frac Nw}(1-p^{1-2s}),
			\]
			and
			\[
				\rho_{\mathfrak{a}}(s,m) = \lt(\frac{1}{wN}\rt)^s \sum_{(c,N/w)=1} \frac{1}{c^{2s}} \sum_{d \text{ (mod }c w)^*}e\lt(-\frac{md}{c w}\rt).
			\]
		\end{prop}


\section{The Amplified Sum}\label{sec:amplification}

	From the approximate functional equation \eqref{eq:approximateFunctional}, we obtain the bound
	\[
		L\lt(\hf,f,\chi\rt) \ll \lt|\sum_{n} \frac{A(n)\chi(n)}{\sqrt{n}}V\lt(\frac{2\pi n}{\sqrt{N}Q}\rt)\rt| + \lt|\frac{1}{g(1,\overline{\chi})} \sum_{n=1}^\infty \frac{A(n)g(n,\chi')}{\sqrt{n}}V\lt(\frac{2\pi n}{\sqrt{N}Q}\rt)\rt|. 
	\]
	Applying summation by parts and then a smooth partition of unity we find that for a smooth function $H$ supported on the interval $[1,2]$,
	\begin{align}\label{eq:shortsum}
		L\lt(\hf,f,\chi \rt) &\ll (\sqrt{N}Q)^{-\hf} \max_{x\ll (\sqrt{N}Q)^{1+\vep}} \lt|\sum_m A(m) \chi(m) H\lt(\frac mx\rt)\rt|\\
		& \aquad+ (\sqrt{N}Q)^{-\hf} \max_{x\ll (\sqrt{N}Q)^{1+\vep}} \lt|\frac{1}{g(1,\overline{\chi})} \sum_m A(m) g(m,\chi') H\lt(\frac mx\rt)\rt|.\notag
	\end{align}
	
	In order to bound the quantity inside the absolute value , we average over all characters $\psi$ modulo $Q$, amplifying the term with $\psi = \chi'$ as in \cite{duke2002subconvexity}. Consider the sum,
	\begin{equation}\label{eq:amplifiedSum}
		S := \sum_{\psi \pmod Q} \lt|\sum_{\substack{\ell \text{ prime}, L\leq \ell \leq 2L\\ (\ell,N) =(\ell,Q) = 1}} \psi(\ell) \overline{\chi'(\ell)} \rt|^2 \lt|\frac{1}{g(1,\overline{\chi})} \sum_{m=1}^\infty A(m)g(m,\psi)H\lt(\frac mX\rt)\rt|^2,
	\end{equation} 
	where $L$ is a large quantity.

	One expects that for $\psi \neq \chi'$, the summation over the primes $\ell$ has a lot of cancellation and is small compared to the $\psi  = \chi'$ term, which is a sum of about $L/\log L$ many $1$'s. The beauty of the technique is that we do not need to demonstrate the cancellation  to be able to take advantage of it. 
	
	\begin{prop}\label{prop:boundSWithShiftedSums}
		Let $\chi,\chi'$ be two characters modulo $Q$ with $\chi$ primitive. Let $H\in C_c^\infty (\R)$ with $\operatorname{supp}(H) \subset [1,2]$. Put $S$ as in \eqref{eq:amplifiedSum} and for $\ell_1, \ell_2$ define, $S_i = S_i(\ell_1,\ell_2)$ for $i = 1,2,3$ as follows: 
		\begin{align*}
			S_1 &= \sum_{\substack{m_1,m_2\\ m_1\ell_1= m_2\ell_2}} A(m_1)\overline{A(m_2)} H\lt(\frac{m_1}{X}\rt)H\lt(\frac{m_2}{X}\rt)\\
			S_2 &= \sum_{\substack{m_1,m_2,h>0\\m_1\ell_1 = m_2\ell_2 + hQ}}   A(m_1)\overline{A(m_2)} H\lt(\frac{m_1}{X}\rt)H\lt(\frac{m_2}{X}\rt)\\
			S_3 &= \sum_{\substack{m_1,m_2,h>0\\m_1\ell_1 + hQ = m_2\ell_2}}   A(m_1)\overline{A(m_2)} H\lt(\frac{m_1}{X}\rt)H\lt(\frac{m_2}{X}\rt).
		\end{align*}
		The following inequality is satisfied:
		\begin{equation}\label{eq:boundSWithShiftedSums}
		S \leq \phi(Q)\sum_{\substack{L \leq \ell_1,\ell_2\leq 2L \\ \ell_1, \ell_2 \text{ primes}}} \chi'(\ell_1)\overline{\chi'(\ell_2)} (S_1 + S_2 + S_3).
		\end{equation}
		Taking all but one term in the summation $S$ this inequality implies, 
		\begin{align}\label{eq:boundForSingleSummand}
			&\frac{L^2}{(\log L)^2} \lt|\frac{1}{g(1,\overline{\chi})} \sum_{m=1}^\infty A(m) g(m,\chi') H\lt(\frac mX\rt)\rt| \\
			&\bquad\aquad\leq \phi(Q)\sum_{\substack{L \leq \ell_1,\ell_2\leq 2L \\ \ell_1, \ell_2 \text{ primes}}} \chi'(\ell_1)\overline{\chi'(\ell_2)} (S_1 + S_2 + S_3)\notag
		\end{align}
	\end{prop}

	\begin{remark}
			Notice that when $\chi'$ is taken to be the primitive character $\overline{\chi}$, we get exactly the same bound \eqref{eq:boundForSingleSummand} for 
		\[
			\lt|\sum_{m=1}^\infty A(m) \chi(m) H\lt(\frac mX\rt)\rt|;
		\]
		the other summand in \eqref{eq:shortsum}.
	\end{remark}
	\begin{proof}
		Apply the Parseval equality in the group $(\Z/Q\Z)^*$ which states that for a function $F(\psi)$ on characters modulo $Q$, 
		\[
			\sum_{\psi \text{ mod} Q} |F(\psi)|^2 = \sum_{\substack{a \text{ mod }Q\\ (a,Q)=1}} |\hat{F}(a)|^2,
		\]
		where $$\hat{F}(a) = \frac{1}{\sqrt{\phi(Q)}}\sum_{\psi \text{ mod }Q} F(\psi)\overline{\psi(a)}.$$

		We apply this formula to $$F(\psi) =  \frac{1}{g(1,\overline{\chi})} \sum_{m} \sum_{\substack{L\leq \ell \leq 2L\\ \ell \text{ prime}\\ (\ell,NQ)=1}}  A(m)g(m,\psi) H\lt(\frac mX\rt) \overline{\chi'(\ell)}\psi(\ell).$$ Then for $(a,Q)=1$,
		\begin{align*}
			\hat{F}(a) &=\frac{1}{g(1,\overline{\chi})} \frac{1}{\sqrt{\phi(Q)}}\sum_{\psi \text{ mod }Q} \sum_{m=1}^\infty\sum_{\substack{L \leq \ell\leq 2L\\ \ell \text{ prime}\\ (\ell,NQ)=1}} \sum_{u \text{ mod }Q}\\
			&\bquad A(m)\psi(u) e\lt(\frac{mu}{Q}\rt) H\lt(\frac{m}{X}\rt) \overline{\chi'(\ell)}\psi(\ell) \overline{\psi(a)} \\
			&= \frac{\sqrt{\phi(Q)}}{g(1,\overline{\chi})} \sum_m \sum_{\substack{L \leq \ell \leq 2L\\ \ell \text{ prime},(\ell,NQ)=1}} \sum_{\substack{u \text{ mod }Q\\ u\ell \equiv a \text{ mod }Q}} A(m) e\lt(\frac{mu}{Q}\rt) H\lt(\frac{m}{X}\rt) \overline{\chi'(\ell)},
		\end{align*}
		which we obtained after summing over $\psi \mod Q$ due to orthogonality relations.

		Applying Parseval's formula, and then expanding the sum over relatively prime $a$ to all such ones, and opening up the squares,
		\begin{align*}
			S&= \sum_{\substack{a \text{ mod }Q\\ (a,Q)=1}} |\hat{F}(a)|^2\\
			& = \sum_{\substack{a \text{ mod }Q\\ (a,Q)=1}} \frac{\phi(Q)}{Q} \lt|\sum_{m=1}^\infty \sum_{\substack{L\leq \ell \leq 2L\\ \ell \text{ prime},(\ell,N)=1}} \sum_{\substack{u \text{ mod }Q\\ u\ell = a \text{ mod }Q}} A(m)e\lt(\frac{mu}{Q}\rt) H\lt(\frac{m}{X}\rt) \overline{\chi'(\ell)} \rt|^2 \displaybreak[0]\\
			& \leq \frac{\phi(Q)}{Q} \sum_{a \text{ mod }Q} \lt|\sum_m \sum_{\substack{L\leq \ell \leq 2L\\ \ell \text{ prime}, (\ell,NQ)=1\\ }} \sum_{\substack{u \text{ mod }Q\\ u\ell = a \text{ mod }Q}} \mspace{-5mu} A(m)e\lt(\frac{mu}{Q}\rt) H\lt(\frac mX\rt) \overline{\chi'(\ell)} \rt|^2 \displaybreak[0] \\
			&= \frac{\phi(Q)}{Q} \sum_{a \text{ mod } Q} \sum_{m_1,m_2,\ell_1,\ell_2} \sum_{ \substack{ u_1,u_2 \text{ mod Q} \\ u_1 \ell_1 \equiv u_2 \ell_2 \equiv a \text{ mod }Q}} \mspace{-10mu}A(m_1)\overline{A(m_2)} e\lt(\frac{u_1m_1}{Q} - \frac{u_2m_2}{Q}\rt)\\
			&\cquad\times H\lt(\frac{m_1}{X}\rt) H\lt(\frac{m_2}{X}\rt) \chi'(\ell_1)\overline{\chi'(\ell_2)}.
		\end{align*}
		We realize that the sum over $u_1, u_2, a$ can be realized as a sum over a single residue system because of the congruence condition. If $\overline{\ell_2}$ denotes the multiplicative inverse of $\ell_2$ modulo $Q$, then we have $u_2 \equiv u_1 \ell_1 \overline{\ell_2}\text{ mod } Q$. This makes the innermost summation (together with the sum over $a$),
		\[
			\sum_{u_1 \text{ mod Q}} e\lt(u_1\frac{m_1- m_2\ell_1 \overline{\ell_2}}{Q}\rt) = \begin{cases} Q &\text{ if } m_1 \ell_2 \equiv m_2 \ell_1 \text{ mod }Q\\ 0& \text{ otherwise.} \end{cases}
		\] 

		Therefore we can express this summation as,
		\begin{equation*}
			S \leq \phi(Q)\sum_{\substack{L \leq \ell_i\leq 2L \\ \ell_i \text{ primes}}} \chi'(\ell_2)\overline{\chi'(\ell_1)}\sum_{\substack{m_1,m_2 \\ m_1\ell_1 \equiv m_2\ell_2 \text{ mod }Q}} A(m_1)\overline{A(m_2)}H\lt(\frac {m_1}{X}\rt)H\lt(\frac{m_2}{X}\rt).
		\end{equation*}

		The inner sum can be split into three, and hence we obtain the shifted convolution sums.
		\begin{equation}\label{eq:threeShiftedSums}
			\underbrace{\sum_{\substack{m_1,m_2\\ m_1\ell_1 = m_2\ell_2}}}_{S_1}  + \underbrace{\sum_{\substack{m_1,m_2, h>0\\ m_1\ell_1 = m_2\ell_2 + hQ }}}_{S_2}   + \underbrace{\sum_{\substack{m_1,m_2, h>0\\ m_1\ell_1 +hQ = m_2\ell_2 }}}_{S_3} A(m_1)\overline{A(m_2)}H\lt(\frac {m_1}{X}\rt)H\lt(\frac{m_2}{X}\rt).
		\end{equation}
	\end{proof}
	
	In dealing with these sums, $S_1$ is easier to deal with, and the sums $S_2$ and $S_3$ are symmetric.


\section{The Diagonal Sum}\label{sec:diagonalSum}

	The diagonal sum $S_1$ is harder to deal with here than in the integral weight case due to the lack of multiplicativity of Fourier coefficients. However we will use the curious fact that half integral weight oldforms do not have the same nebentypus as the one they are induced from. 

	\begin{prop}\label{thm:diagonal}
		We can bound the sum
		\begin{equation}\label{eq:diagonalSum}
			S_1 = \sum_{\substack{m_1,m_2\\ m_1\ell_1 = m_2\ell_2}} A(m_1)\overline{A(m_2)}H\lt(\frac {m_1}{X}\rt)H\lt(\frac{m_2}{X}\rt)
		\end{equation}
		by $S_1 \ll X$ if $\ell_1 = \ell_2$, and by $S_1 \ll X^\hf L $ when $\ell_1 \neq \ell_2$
	\end{prop}

	\begin{proof}
		In case $\ell_1 = \ell_2$ the sum simply becomes 
		\[
			S_1 = \sum_{m=1}^\infty |A(m)|^2 H\lt(\frac mX\rt)^2.
		\]
		We obtain this sum via an inverse Mellin transform of the Rankin Selberg L-function, 
		\[
			S_1 = \frac{1}{2\pi i} \int_{(2)} L(s,f\times f) h^{(2)}(s) X^s \d s,
		\]
		where
		\[
			h^{(2)}(s) = \int_0^\infty H(t)^2 t^s \frac{\d t}{t}
		\]
		is an entire function of $s$, and 
		\[
			L(s,f\times f) = \sum_{m=1}^\infty \frac{|A(m)|^2}{m^s}.
		\]
		 
		Meromorphic properties of this Dirichlet series are obtained via the inner product of $|f(z)|^2 y^k$ with an Eisenstein series. 
		\[
			\mathcal{V}\langle |f(z)|^2y^k, E^{(N)}(z,s)\rangle_{L^2(\Gamma_0(N)\backslash \H)} = (4\pi)^{-(s+k-1)}L(s,f\times f)\Gamma(s+k-1). 
		\]
		Here $E^{(N)}(z,s)$ is the level $N$, weight 0 spectral Eisenstein series. In particular there is a pole at $s = 1$ with residue equal to $\|f y^{\frac{k}{2}}\|_{L^2}^2 \frac{(4\pi)^k}{\Gamma(k)}$. Moving lines of integration we obtain,
		\[
			S_1 \sim \frac{(4\pi)^k}{\Gamma(k)} \|fy^{\frac k2}\|_{L^2}^2 X.
		\]
		
		Consider the case $\ell_1 \neq \ell_2$. Notice that the sum \eqref{eq:diagonalSum} can be obtained via an inverse Mellin transform of the Dirichlet series,
		\[
			L(s;f\times f; \ell_1,\ell_2) = \sum_{m_1\ell_1 = m_2 \ell_2} \frac{A(m_1)\overline{A(m_2)}}{(m_1\ell_1)^s}
		\]
		as follows:
		\begin{equation}\label{eq:diagonalDoubleMellin}
			S_1 = \lt(\frac{1}{2\pi i}\rt)^2 \int_{(2)}\int_{(2)} L(s+w;f\times f;\ell_1,\ell_2) X^{s+w} h(s) h(w) \d s \d w
		\end{equation}
		where $h$ is the holomorphic function which is the Mellin transform of $H$. 
		\[
			h(s) = \int_0^\infty H(t) t^s \frac{\d t}{t}.
		\]
				
		The series $L(s,f\times f;\ell_1,\ell_2)$ can be obtained from an integral of $f$ as follows:
		\[
			\int\limits_{0}^{\infty} \int\limits_0^1  f(\ell_1z)\overline{f(\ell_2z)} y^{s+k-1} \d x \dy = \frac{1}{\ell_1^{k-1}} L(s,f\times f;\ell_1,\ell_2)  \frac{\Gamma(s+k-1)}{(2\pi)^{s+k-1}} 
		\]
		
		The oldform $f(\ell z)$ is a modular form of the same weight as that of $f$, level $N\ell$ and nebentypus $\chi_\ell$ where $\chi_\ell(d) = \kron{\ell}{d}$ is the Kronecker symbol as in \cite{ShimuraHalfInt73}. Unlike the case of full integral weight, the nebentypus has changed. Now if we \emph{fold} the above integral on the left hand side, we obtain the following,
		\[
			\iint_{\Gamma_0(N\ell_1\ell_2)\backslash \H} f(\ell_1z) \overline{f(\ell_2 z)} y^k E^{(N\ell_1\ell_2)}\lt(z,s,\chi_{\ell_1\ell_2}\rt) \dz.
		\] 
		
		Since $\ell_1 \neq \ell_2$ the character is nontrivial, and therefore the Eisenstein series does not have any poles in the region $\Re(s) \geq \hf$. Making a change of variables $v = s+w$ and moving the $v$ line of integration in \eqref{eq:diagonalDoubleMellin} we obtain,
		\[
			S_1 = \lt(\frac{1}{2\pi i}\rt)^2\int_{(\hf)}\int_{(2)} L(v,f\times f;\ell_1,\ell_2) h(s) h(v-s) X^v \d v \ ds.
		\] 
		
		Now there could be $\ell$ dependence in the integral over the line $\Re(s) = \hf$. Let us bound this dependence. We will ignore the dependence on $\Im(s)$, because the functions $h$ in the above integral ensure convergence in that aspect. Call $\Gamma = \Gamma_0(N\ell_1\ell_2)$. Using the Cauchy-Schwartz inequality,		
		\begin{align*}
			L(\tfrac12 + it;f\times f;\ell_1,\ell_2 )  &\ll  \ell_1^{k-1}\lt(\iint_{\Gamma \backslash\H} |f(\ell_1z)|^2 y^k \dz\rt)^\hf \\
			&\times\lt(\iint_{\Gamma\backslash \H}\lt|f(\ell_2z)\rt|^2y^k \lt|E^{(N\ell_1\ell_2)}(z,\tfrac 12 + it,\chi)\rt|^2\dz\rt)^{\mspace{-4mu}\hf}\mspace{-7mu}.
		\end{align*}
		The first integral can be bounded by the sup norm of the integrand times the volume of the domain, which makes the first factor $\ll \sqrt{\ell_1\ell_2}\ell_1^{-k/2}\|f y^\frac k2\|_\infty$.
		
		For the second factor we use the inequality in \cite{blomer2012period} for the Eisenstein series. We use the Fourier expansion of the Eisenstein series $E^{(N\ell_1\ell_2)}(z,s,\chi_{\ell_1\ell_2})$ in order to bound it in each of the neighborhoods of various cusps in $\Gamma_0(N\ell_1,\ell_2)\backslash\H$. We can obtain a Fourier expansion of the Eisenstein series from the explicit expression of Eisenstein series in \cite{duke2002subconvexity}, section 7. With our assumption on the level, all cusps $\mathfrak{a}$ are equivalent to one of $1/w$ under $\Gamma_0(N)$ for some $w |N$. We have that,
		\[
			E^{(N\ell_1\ell_2)}( \sigma_{\mathfrak{a}}z,s,\chi_{\ell_1\ell_2}) = \lt(\frac{N\ell_1\ell_2}{w}\rt)^s \frac{1}{L(2s,\chi_{\ell_1\ell_2})} \sum_{\substack{(c,d) \neq (0,0)\\ (c,N/w) = 1}} \frac{\chi_{\ell_1\ell_2}(d)y^s}{|cNz + d|^{2s}}.
		\]
		The Fourier expansion of the Eisenstein series obtained with the explicit expression above implies that,
		\[
			E^{(N\ell_1\ell_2)}(\sigma_{\mathfrak{a}}z,\tfrac12 + it,\chi_{\ell_1\ell_2}) \ll \begin{cases} \sqrt{\ell_1\ell_2} y^{\hf + \vep} &\text{if } \mathfrak{a} = \infty\\ \sqrt{\ell_1\ell_2} &\text{otherwise}, \end{cases}
		\] 
		for $ y \gg 1$. Summing all the contributions from the cuspidal regions we get that the second integral is less than,
		\[
			N^\vep \ell_1\ell_2 \iint_{\Gamma\backslash \H} |f(\ell_2z)|^2 y^k  y^{1+ \vep}\dz \ll \ell_1\ell_2 L(1+\vep, f\times f) \ell_2^{-k - \vep}.
		\]
		Notice that the number of cusps stays constant and does not depend on the size of $\ell_i$. The second factor contributes $\ll L^{1-k/2}$. Combining the two, we obtain the result.
	\end{proof}
	
	Notice for the expression of $S$ as \eqref{eq:boundSWithShiftedSums} the $S_1$ contribution of the kind coming from $\ell_1 = \ell_2$ is $L$ many, and of the second kind is $L^2$-many. Therefore the contribution in $S$ coming from $S_1$ is $\ll Q(XL + X^\hf L^3)$.
	

\section{The Double Dirichlet Series} \label{sec:doubleDirichlet}

	We now express $S_2$ and $S_3$ as inverse Mellin integral of a double Dirichlet series, whose analytic properties will give us a bounds for the shifted convolution sums.

	Given positive integers $h$ and $Q$, we create the double Dirichlet series,
	\begin{equation}\label{eq:doubleDirichlet}
		Z_{Q,\ell_1,\ell_2}(s,w) = Z_Q(s,w) := \sum_{\substack{m_2,h>0\\ \ell_1m_1 = \ell_2m_2 + hQ}} \frac{A(m_1)\overline{A(m_2)}\lt(1+ \frac{hQ}{\ell_2m_2}\rt)^{\frac{k-1}2 }}{(\ell_2m_2)^s(hQ)^{w+\frac{k-1}{2}}}.
	\end{equation}
	This double Dirichlet series converges for $\Re(s), \Re(w)>1$. Note that this is the same construction as in \cite{JeffShiftedSum}, with the integral weight modular forms replaced by half integer weight ones. The function $Z_Q$ has a meromorphic continuation to all of $\C^2$ as per the methods there. This will be elaborated on in Section \ref{sec:analyticProperties}. 

 	\begin{prop}
		Let $Z_Q(s,w)$ be defined as above in \eqref{eq:doubleDirichlet}, and $S_2$ be as above in equation \eqref{eq:threeShiftedSums}
		\begin{align}
			S_2 &= \lt(\frac{1}{2\pi i}\rt)^3 \int\limits_{(\alpha)}\int\limits_{(\beta)} \int\limits_{(\gamma)} Z_Q\lt(s+w-u,u-\tfrac{k-1}{2}\rt) \notag \\
			&\bquad \times\frac{\Gamma\lt(w+\frac{k}{2}-u\rt)\Gamma(u)}{\Gamma\lt(w+\frac{k}{2}\rt)} h(s)h(w) \ell_2^s \ell_1^w X^{s+w} \du \dw \ds. \label{eq:tripleMellin}
		\end{align}
		where $\alpha = (k+1)/2, \beta = 1+2\vep$ and $\gamma = (k+1)/2+\vep$ for some arbitrarily small positive constant $\vep$.
	\end{prop}

	\begin{proof}
		Choose $s,w$ and $u$ to be three complex variables. Let $\alpha = \Re(s) = (k+1)/2, \beta = \Re(w) = 1+2\vep$. Put $\gamma = \Re(u) = (k+1)/2+\vep$. Note that with these choices one has $\Re(s+w-u)>1$ and $\Re(u - (k-1)/2) >1$. Furthermore, $\Re(w + k/2) > \Re(u) = \gamma$, which allows us to use the formula 6.422\#3 in Gradstehyn-Rhyzik \cite{GR}:
		\[
			\frac{1}{2\pi i} \int_{(\gamma)} \frac{\Gamma (u)\Gamma(z - u)}{\Gamma(z)} t^{-u} \d u = (1+t)^{-z}
		\]
		for $\gamma$ such that $0 < \gamma \leq \Re(z)$ and $\arg(t) <\pi$.
		
		\begin{align*}
			\frac{1}{2\pi i}& \int_{(\gamma)} Z_Q(s+w-u,u-\tfrac{k-1}{2})\frac{\Gamma\lt(w+\frac{k}{2}-u\rt)\Gamma(u)}{\Gamma \lt(w+\frac{k}{2}\rt)} \du \\
			&= \mspace{-35mu}\sum_{\substack{m_2,h>0\\ \ell_1m_1 = \ell_2m_2 + hQ}}\mspace{-35mu} \frac{A(m_1)\overline{A(m_2)}\lt(1\mspace{-5mu} + \mspace{-5mu} \tfrac{hQ}{\ell_2m_2}\rt)^{\mspace{-5mu}\frac{k-1}2 }}{(\ell_2m_2)^{s+w}} \frac{1}{2\pi i} \mspace{-5mu}\int\limits_{(\gamma)} \mspace{-5mu} \frac{\Gamma\lt(w+\tfrac{k}{2}-u\rt)\Gamma(u)}{\Gamma\lt(w+\frac{k}{2}\rt)}\mspace{-3mu}\lt(\mspace{-3mu}\frac{\ell_2m_2}{hQ}\mspace{-3mu}\rt)^u \mspace{-3mu}\du \\
			&= \mspace{-35mu}\sum_{\substack{m_2,h>0\\ \ell_1m_1 = l_2m_2 + hQ}}  \frac{A(m_1)\overline{A(m_2)}}{(\ell_2m_2)^{s}(\ell_1m_1)^{w}}.
		\end{align*}
		
		Now we integrate this double Dirichlet series in both $s$ and $w$ variables.
		\begin{align*}
			&\lt(\frac{1}{2\pi i}\rt)^3 \int_{(\alpha)}\int_{(\beta)} \int_{(\gamma)} Z_Q\lt(s+w-u,u-\frac{k-1}{2}\rt) \\
			&\bquad \qquad\times\frac{\Gamma\lt(w+\frac{k}{2}-u\rt)\Gamma(u)}{\Gamma\lt(w+\frac{k}{2}\rt)} h(s)h(w) \ell_2^s \ell_1^w X^{s+w} \du \dw \ds \\
			&= \lt(\frac{1}{2\pi i}\rt)^2 \int_{(\alpha)}\int_{(\beta)}\mspace{-25mu} \sum_{\substack{m_2,h>0\\ \ell_1m_1 = \ell_2m_2 + hQ}} \mspace{-15mu}\frac{A(m_1)\overline{A(m_2)}}{(\ell_2m_2)^{s}(\ell_1m_1)^{w}}(X\ell_2)^s (X\ell_1)^w h(s)h(w)\dw\ds \displaybreak[0]\\
			&= \sum_{\substack{m_2,h>0\\ \ell_1m_1 = \ell_2m_2 + hQ}} A(m_1)\overline{A(m_2)}H\lt(\frac{m_1}{X}\rt)H\lt(\frac{m_2}{X}\rt). 
		\end{align*}
	\end{proof}

	We are able to obtain the shifted convolution sums as a triple inverse Mellin integral of the double Dirichlet series we introduced. All that is left for us to do is analyze the analytical properties of $Z_Q(s,w)$ in both of the variables and move lines of integration to better understand the asymptotic behaviour of $S_2$ and $S_3$ in $X$ and $\ell$ variables.


\section{The analytic properties of $Z_Q$}\label{sec:analyticProperties}

	Our definition of $Z_Q$ will be slightly more general than the previous section for notational convenience. Let $f$ and $g$ be two modular forms of the same weight $k$ (half integer). Further assume,
	\begin{align*}
		f(z) &= \sum_{n=1}^\infty A(n)n^{\frac{k-1}2}e(nz) = \sum_{n=1}^\infty a(n)e(nz) \qquad \text{ and } \\
		g(z) &= \sum_{n=1}^\infty B(n)n^{\frac{k-1}{2}}e(nz) = \sum_{n=1}^\infty b(n)e(nz).
	\end{align*}
	We define the shifted convolution Dirichlet series of $f$ and $g$ as follows,
	\begin{equation}\label{eq:shiftedDirichlet}
		D(s;h) := D_{f,g}(s;h) = \sum_{n=1}^\infty \frac{a(n+h)\overline{b(n)}}{n^{s+k-1}}.
	\end{equation}
	Note that if $f$ and $g$ are oldforms given by $f(z) = f_0(\ell_1 z)$ and $g(z) = g_0(\ell_2 z)$, then this sum becomes 
	\[
	D_{f,g}(s,h) = \sum_{\substack{m_2 > 0\\ m_1\ell_1= m_2\ell_2 + h}} \frac{a_0(m_1)\overline{b_0(m_2)}}{(m_2\ell_2)^{s+k-1}},
	\]
	where $a_0(n)$ and $b_0(n)$ are the $n$\th\ Fourier coefficients of $f_0$ and $g_0$ respectively.
	
	\begin{defn}\label{def:ZQ}
		Let $f$ and $g$ be two modular forms of the same weight. Call the double Dirichlet series 
		\[
			Z_{Q,f,g}(s,w) = \sum_{h=1}^\infty \frac{D_{f,g}(s,hQ)}{(hQ)^{s+\frac{k-1}{2}}} = \sum_{h = 1}^\infty \sum_{n=1}^\infty \frac{a(n+hQ) \overline{b(n)}}{n^{s+k-1}(hQ)^{w+\frac{k-1}{2}}}.
		\]
		We might, by abuse of notation also drop some of the indices and simply denote $Z_Q = Z_{Q,f,g}$. 
	\end{defn}
	We hope this notation does not cause any confusion over \eqref{eq:doubleDirichlet}. These two double Dirichlet series are compatible, and to compare notations, one has;
	\begin{equation}\label{eq:twoDifferentZs}
		Z_{Q,f,\ell_1,\ell_2} = (\ell_1\ell_2)^{\frac{1-k}{2}}Z_{Q, f(\ell_1\cdot), f(\ell_2 \cdot)}.
	\end{equation}


	\subsection{The Linear Functional}\label{subsec:linearFunctional}	
		Let $\Gamma \subset \SL_2(\Z)$ be any congruence subgroup. We can express the shifted sum Dirichlet series $D_{f,g}(s,h)$ as the image of some conjugate-linear functional defined in some dense subspace of $L^2(\Gamma\backslash\H)$. We describe the functional first. This is going to be a holomorphic section of functionals $\wp_h(\cdot,s)$ defined on some subspace of the space of square integrable functions and some open subset of complex number $s$.
		\begin{defn}
			Define the conjugate-linear functional $\wp_h(\cdot,s)$ by the integral
			\begin{align}\label{eq:linearFunctional}
				L^2(\Gamma\backslash \H) &\stackrel{\wp_h}{\longrightarrow} \C\\
				V(z) = \sum_{n\in\Z} C(n,y) e(nx) &\mapsto \int_0^\infty \overline{C(-h,y)}e^{2\pi h y} y^{s-1} \dy\notag,
			\end{align}
			on the subspace of square integrable functions where integral converges. Here $C(n,y)$ is the $n$\th\ Fourier coefficient of the automorphic function $V(z)$.
		\end{defn}
		Notice that this functional is not defined on all of $L^2(\Gamma \backslash\H)$, but on the subspace of functions whose $h$\th\ Fourier coefficient has decay $o(e^{-2\pi hy})$. 
		\begin{lemma}
			Let $V(z) \in L^2(\Gamma\backslash \H)$, and let $C(h,y)$ be the $h$\th\ Fourier coefficient of $V$. Assume that 
			\[
				C(h,y)e^{2\pi i hy}y^A = O(1) \text{ as } y\to \infty \qquad\text{ and } \qquad \frac{C(h,y)}{y^B} = O(1) \text{ as } y \to 0,
			\]
			for some positive real constants $A$ and $B$. Then the integral defined by $\wp_h(V,s)$ converges for all $s$ with $1-B<\Re(s)<1+A$.
		\end{lemma}
	
		The raison d'\^{e}tre of this functional is the following computation. Given two modular forms $f$ and $g$, automorphic under the same congruence subgroup $\Gamma$ with the same weight $k$, the function $$V(z)  := f(z)\overline{g(z)}y^k \in L^2(\Gamma\backslash \H)$$ has sufficient decay in its $h$\th Fourier coefficient and thus is in the domain of the functional for all $s$ with $\Re(s)>1$ and upon calculation one sees that
		\[
			\wp_h(V,s) = \sum_{n=1}^\infty \frac{a(n+h)\overline{b(n)}}{n^{s+k-1}} \Gamma(s+k-1)(4\pi)^{-(s+k-1)}.
		\]
		Let us call
		\[
			D^*(s;h) := \Gamma(s+k-1)(4\pi)^{-(s+k-1)} D_{f,g}(s;h),
		\]
		where the star indicates some sort of ``completed'' Dirichlet series, in the sense that the Gamma factors come naturally out of some integral. Thus with the above notation $\wp_h(V,s) = D^*_{f,g}(s;h)$.
		

	\subsection{The Spectral Expansion} \label{subsec:spectralExpansion}
	
		We would like to spectrally expand the shifted sum Dirichlet series $D_{f,g}(s;h)$, or more correctly, the linear functional $\wp_h(\cdot,s)$. What we mean by that is an expansion of the form
		\begin{equation}\label{eq:spectralExpansion}
			\wp_h(V,s) \!=\! \sum_{j} \!\wp_h(u_j,s)\langle u_j, V \rangle + \frac{1}{4\pi} \! \sum_{\mathfrak{a}}\!\! \int\limits_{-\infty}^\infty \!\wp_h(E_\mathfrak{a}(*,\tfrac12 + it),s)\langle E_\mathfrak{a}(*, \tfrac12 + it), V \rangle \d t.
		\end{equation}
		This expansion is obtained by expanding the square integrable function $V$ in a basis of eigenfunctions of the Laplacian, and naively using the conjugate-linearity of $\wp_h$. There is a slight problem with this plan. Firstly the Maass forms and the Eisenstein series are only in the domain of the functional for $\tfrac12<\Re(s) < 1$, which does not overlap with the region of convergence of the Dirichlet series $D(s;h)$. A second problem is that the linear functional is not continuous with respect to the $L^2$-norm, and therefore after applying the spectral expansion to $V \in L^2$ we cannot simply take the sums and integrals outside of $\wp_h$. These two problems have been overcome in \cite{JeffShiftedSum}.
		
		The linear functional $\wp_h$ can be realized as an inner product with a Poincare series. Put
		\[
			P_h(z,s) = \sum_{\gamma \in \Gamma_\infty \backslash \Gamma} \Im(\gamma z)^s e^{-2\pi i h \gamma z},
		\]
		and then 
		\[
			\wp_h(V,s) = \mathcal{V}\langle P_{h}(*,s), V\rangle.
		\]
		Note that this Poincare series is far from $L^2$, it is exponentially growing in $y$. This is the same reason the functional $\wp_h$ does not have all of $L^2$ as its domain, nor does the domain include Maass forms. In order to overcome this problem we modify this Poincare series to have slightly less growth and therefore enlarge the domain of convergence of the functional. For any $\delta >0$, put
		\[
			P_h(z,s;\delta) = \sum_{\gamma \in \Gamma_\infty \backslash \Gamma} \Im(\gamma z)^s e^{-2\pi i h \Re(\gamma z)}e^{2\pi h \Im(\gamma z) (1-\delta)}.
		\]
		This perturbation of the Poincare series, and hence the functional, allows us to include the Laplace eigenvalues and Eisenstein series into the domain. 
		
		The second problem is easier to overcome. One introduces a truncation parameter $Y$ for the Poincare series, thus making it compactly supported, a fortiori square integrable. Then we can safely take the spectral decomposition for $P$ (or equivalently $V$) in the inner product. The Poincare series in \cite{JeffShiftedSum} is 
		\begin{equation}\label{eq:poincareSeries}
			P_{h,Y}(z,s,\delta) := \sum_{\gamma\in\Gamma_\infty \backslash \Gamma} \Im(\gamma z)^s \mathbbm{1}_{[Y\inv, Y]} (\Im \gamma z) e^{2\pi m \Im(\gamma z)(1-\delta)}e(-m\Re(\gamma z)) .
		\end{equation}

		In the calculation of $\langle P_{h,Y}(*,s,\delta), \phi\rangle $ for $\phi$ a Maass form or an Eisenstein series the following function will make an appearance.
		\begin{defn}
			For any complex $t$ and $\delta >0$ and $\Re(s)> \tfrac12 + \Re(t)$ define
			\[
				M(s,t,\delta) := \int_0^\infty e^{(1-\delta)y} K_{it}(y) y^{s-\hf} \dy.
			\]
		\end{defn}
		This function has a meromorphic continuation to the entire complex plane and analytic properties of $M(s,t,\delta)$ have been carefully studied in \cite{JeffShiftedSum}. We cite Proposition 3.1 loc.\ cit.\ .
		
		\begin{prop}[Hoffstein, Hulse]\label{prop:Mfunction}
			Let $M(s,t,\delta)$ be defined as above and $\vep, \delta >0$ and $A \gg 1$ not an integer. Then the function has meromorphic continuation to all of $\C$ with simple poles at $s = \tfrac12 \pm it - r$ for any nonnegative integer $r \in \N$ and $t \in \C^*$, with residue,
			\begin{equation*}
				\Res_{s = \hf \pm it - r}\mspace{-17mu} M(s,t,\delta) \mspace{-4mu} =\mspace{-4mu} \frac{(-1)^r\sqrt{\pi} 2^{r\mp it}\Gamma(\tfrac12 \mspace{-3mu}\mp\mspace{-3mu} it\mspace{-4mu}+\mspace{-4mu}r)\Gamma(\pm 2it \mspace{-4mu}-\mspace{-4mu} r)}{r!\Gamma(\tfrac12 + it)\Gamma(\tfrac12 - it)} + O\lt((1\mspace{-4mu}+\mspace{-4mu}|t|)^re^{-\frac{\pi}{2}|t|}\delta\rt)\mspace{-5mu}.
			\end{equation*}
			
			If $t = 0$, there is this time a double pole at $s = \tfrac12 - r$, with Laurent series expansion at that point being,
			\begin{align*}
				M(s,t,\delta) = \frac{c_2(r) + O_r(\delta)}{(s - \tfrac 12 +r)^2} + \frac{c_1(r) + O_r(\delta)}{(s- \tfrac12 + r)} + O_r(1) + O_r(\delta^{1-\vep}).			
			\end{align*}
			
			The function itself can be bounded from above. Put $s = \sigma +i\tau$, assume it is at least an distance of $\vep>0$ away from all the poles of $M(s,t,\delta)$. There exists an $A_0 > 1 + |\sigma| + |\Im(t)|$ such that for all $A>A_0$ 
			\begin{equation}\label{eq:MbadBound}
				M(s,t,\delta) \ll_{A,\vep} \delta^{-A} (1+|t|)^{2\sigma -2-2A} (1+|\tau|)^{9A}e^{-\frac{\pi}{2}|\tau|}.
			\end{equation}
			
			One also has a bound which converges as $\delta$ goes to zero. Again with the conditions as above, but this time with $\delta(1+|t|)^2<1$, 
			\begin{align}\label{eq:MgoodBound}
				M(s,t,\delta) =& \frac{\sqrt{\pi}2^{\hf-s}\Gamma(s-\tfrac12 + it) \Gamma(s - \tfrac12 - it)\Gamma(1-s)}{\Gamma(\tfrac12 - it)\Gamma(\tfrac12+it)} \\
				&+ O_{A,\sigma,\vep}\lt((1+|t|)^{2\sigma-2+2\vep}(1+|\tau|)^{9A}e^{-\frac{\pi}{2}|\tau|}\delta^{\vep}\rt)\notag.
			\end{align}
		\end{prop}
		
		\begin{proof}[Some Remarks]
			The complete proof of the proposition can be found in \cite{JeffShiftedSum} and \cite{TomsThesis}. The proof of \eqref{eq:MbadBound} is missing from the following discussion. The meromorphic continuation of $M$ can be seen from direct evaluation of the integral using formula 6.621 \# 3 of \cite{GR} after which one gets,
			\[
				M(s,t,\delta) = \frac{\sqrt{\pi}2^{it}}{\delta^{s-\hf + it}} \frac{\Gamma(s\mspace{-1mu}-\mspace{-1mu}\tfrac 12 \mspace{-1mu}-\mspace{-1mu}it) \Gamma(s\mspace{-1mu}-\mspace{-1mu}\tfrac12 \mspace{-1mu}+ \mspace{-1mu}it)}{\Gamma(s)} \,\tensor[_2]{F}{_1} \!\!\!\lt(s\mspace{-2mu}-\mspace{-2mu}\frac12 \mspace{-2mu}+ \mspace{-2mu}it, \frac12 \mspace{-2mu}+\mspace{-2mu} it;s;1\mspace{-2mu}-\mspace{-2mu}\frac{2}{\delta}\rt)
			\]
			where $\tensor[_2]{F}{_1}\!\! =F$ is the classical Gauss hypergeometric function. The integral evaluates to this function under the conditions $\delta >0$ and $\Re(s) > |\Im(t)|+\tfrac12$. However after the evaluation we realize that the right hand side immediately gives the meromorphic continuation of $M$ to the entire plane for any $\delta >0$. For the behavior of $M(s,t,\delta)$ near $\delta = 0$, one could apply a Mellin-Barnes integral as in the original proof, or one could use a linear transformation formula of the Gauss hypergeometric function, such as formula 15.3.8 of \cite{StegunHandbook} after which we obtain 
			\begin{align}
				M(s,t,\delta) \mspace{-3mu}=& \frac{\sqrt{\pi}}{2^{s-\hf}}\frac{\Gamma(s\mspace{-3mu}-\mspace{-3mu}\hf \mspace{-3mu}+ \mspace{-3mu}it )\Gamma(s\mspace{-3mu}-\mspace{-3mu}\hf\mspace{-3mu} -\mspace{-3mu}it)\Gamma(1\mspace{-3mu}-\mspace{-3mu}s)}{\Gamma(\hf+it)\Gamma(\hf-it)} F\mspace{-3mu}\lt(\mspace{-3mu}s\mspace{-2mu}-\mspace{-2mu}\hf\mspace{-2mu}+\mspace{-2mu}it, s\mspace{-2mu}-\mspace{-2mu}\hf \mspace{-2mu}-\mspace{-2mu} it;s;\frac{\delta}{2} \rt)\notag\\ 
				&+\sqrt{\frac\pi2} \Gamma(s-1)\delta^{1-s}F\lt(\hf+it,\hf-it;2-s; \frac\delta2\rt). \label{eq:MasSumOfTwo}
			\end{align}
			
			Note that for $\Re(s)<1$, we may send $\delta \to 0$, and then the second summand vanishes.
			\[
				M(s,t,0) = \frac{\sqrt{\pi}}{2^{s-\hf}} \frac{\Gamma(s-\hf + it ) \Gamma(s-\hf -it)\Gamma(1-s)}{\Gamma(\hf+it)\Gamma(\hf-it)}.
			\]
			From this expression it might seem like there is a pole at $s =1$ however note that this is outside of the allowed domain. For any positive $\delta$, there are poles at $s = 1$ in either summand of \eqref{eq:MasSumOfTwo}, whose residues cancel one another.
		\end{proof}

		A simple calculation, and the definition of the function $M$, gives the following proposition.
		\begin{prop}\label{prop:PinnerProduct}
			Let $u_j$ be a Maass form of eigenvalue $\tfrac14 + t_j^2$ as in \eqref{eq:FourierMaass}, and $E_\mathfrak{a}(z,u)$ be an Eisenstein series with Fourier expansion expansion given in Proposition \ref{prop:eisensteinFourier}. Put $\mathcal{V}= \Vol(\Gamma \backslash \H)$.
			\[
				\lim_{Y \to \infty } \mathcal{V} \langle P_{h,Y}(*,s,\delta), u_j\rangle = \frac{\overline{\rho_j(-h)}}{(2\pi h)^{s-\hf}}M(s,t_j,\delta) 
			\]
			and
			\[
				\lim_{Y \to \infty} \mathcal{V}\langle P_{h,Y}(*,s,\delta), E_{\mathfrak{a}}(*,u) \rangle =\frac{\overline{\rho_{\mathfrak{a}}(u,-h)}}{(2\pi h)^{s-u - 1}} \frac{4\pi^u}{\Gamma(u)}M(s,\tfrac{1}{i}(u-\tfrac12),\delta).
			\]
		\end{prop}

		Next we sum and integrate over the eigenvalues to obtain an expansion such as \eqref{eq:spectralExpansion}. First we give some definitions.
		\begin{defn}
			Let $f$ and $g$ be two holomorphic modular forms of level $N$, weight $k$. Put, $V(z) = f(z)g(z)y^k$ as before. Define,
			\begin{equation}\label{eq:innerProductPV}
				D_{f,g}^*(s;h;\delta) := \lim_{Y\to\infty } \mathcal{V} \langle P_{h,Y}(*,s;\delta), V(z)\rangle.
			\end{equation}
		\end{defn}
		
		We can compute the inner product on the right via an unfolding method and we get for $\Re(s)>1$ a convergent Dirichlet series.
		\[
			D_{f,g}^*(s;h;\delta) := \frac{\Gamma(s+k-1)}{(4\pi)^{s+k-1}}\sum_{n=1}^\infty \frac{a(n+h) \overline{b(n)}}{(n + h\delta/2)^{s+k-1}}.
		\] 
		
		\begin{defn}
			Set
			\[
				D_{f,g}(s;h;\delta) := \sum_{n=1}^\infty \frac{a(n+h) \overline{b(n)}}{(n + h\delta/2)^{s+k-1}},
			\]
			which converges for $\Re(s)>1$.
		\end{defn}
		
		Notice that 
		\[
			\lim_{\delta \to 0} D_{f,g}(s; h; \delta) = D_{f,g}(s; h).
		\]
		
		When we expand $P_{h,Y}(z,s,\delta)$ in a spectral basis we obtain a spectral expansion of $D_{f,g}^*(s;h;\delta)$ as in \eqref{eq:DdeltaSpectral} below. First let us settle some convergence issues.
		
		Proposition \ref{prop:PinnerProduct} shows that the $-h$\th\ Fourier coefficient of a Maass form is going to be in the spectral expansion of $D_{f,g}$. When we normalize $u_j$ to have $\|u_j\|_{L^2}=1$, this Fourier coefficient grows exponentially with the eigenvalue. However as in \cite{JeffShiftedSum}, equation (4.3) we have the bound 
		\[
			\sum_{ T\leq|t_j| \leq2T} |\rho_j(-h)|^2  e^{-\pi|t_j|} \ll_h NT^2.
		\] 
		This rapid growth of Fourier coefficients of Maass forms is compensated for by the rapid decay of $\langle u_j, V\rangle$. The proof of the explicit rapid decay is accomplished in the appendix of \cite{JeffShiftedSum} by Andre Reznikov via representation theoretic methods. We cannot apply these mehtods in this context, and throughout sections \ref{sec:AutomorphicKernel} to \ref{sec:selbergTransform} we use methods similar to those in \cite{SarnakIMRN94} proving the following Proposition.
		\begin{prop}\label{prop:tripleInnerBound}
			Given two half integral weight modular forms $f$ and $g$ of weight $k$, put $V(z) = f(z)\overline{g(z)}y^k$. If $u_j$ are Maass forms of type $\hf + it_j$, one has the bound, 
			\[
				\sum_{T\leq |t_j|\leq T+1}\lt| \langle u_j,V \rangle \rt|^2 \ll  \|fy^{\frac k2}\|_{L^\infty}^2 \|gy^{\frac k2}\|_{L^\infty}^2  T^{4k+2}e^{-\pi T}.
			\]
		\end{prop}
		This will be proven in Section \ref{sec:selbergTransform}.
		
		The Poincare series \eqref{eq:poincareSeries} is compactly supported, hence square integrable. Therefore we may apply spectral expansion of $P_{h,Y}(*,s,\delta)$ and then the inner product in \eqref{eq:innerProductPV} and take the limit as $Y\to\infty$. Due to the exponential decay stated in Proposition \ref{prop:tripleInnerBound}, the expression converges. We thus obtain,
		\begin{align}\label{eq:DdeltaSpectral}
			D_{f,g}^*&(s;h;\delta) = \sum_{j} \frac{\overline{\rho_j(-h)}}{(2\pi h)^{s-\hf}} M(s,t_j,\delta)\langle u_j, V\rangle \\
			&+ \frac{1}{4\pi}\sum_{\mathfrak{a}\text{ cusp}} \int\limits_{-\infty}^\infty  \frac{\rho_{\mathfrak{a}}(\tfrac12 - it, -h)}{h^{s-\hf +it}} \frac{2^{\frac32-s}\pi^{1-s -it}}{\Gamma(\tfrac12 - it)} M(s,t,\delta) \langle E_{\mathfrak{a}}(*,\tfrac12 + it),V\rangle \d t.\notag 
		\end{align}
		The convergence of the integral in $t$ requires us to know that $\langle E_{\mathfrak{a}}\lt(*,\tfrac12 + it\rt),V\rangle$ decays exponentially in $t$. We can explicitly calculate this inner product, and the Gamma factors give us the desired exponential decay.
		
		For the convergence of the sum and the integral in \eqref{eq:DdeltaSpectral} we also need decay of $M(s,t,\delta)$ in $t$, as seen in Proposition \ref{prop:Mfunction}. However the rapid decay of $M$ in $t$ is not well behaved as $\delta\to 0$ in \eqref{eq:MbadBound}, therefore using the bound \eqref{eq:MgoodBound} in the same proposition, we are able to get convergence.
		
		Notice that the bound in \eqref{eq:MgoodBound} is polynomially decaying in $t$, and the rate of polynomial decay is faster as $\Re(s)$ becomes smaller. The convergence of the summation over $t_j$ is satisfied only if $s$ has sufficiently small real part. In fact $$\{s\in\C : \Re(s)<-\tfrac14-k\}$$ is the domain for which we can meaningfully take the limit as $\delta \to 0$ in \eqref{eq:DdeltaSpectral}.
		
		Although the integral over the continuous spectrum converges as $\delta \to 0$ in some right half plane, it is not true, not even for a fixed positive $\delta$ that the meromorphic continuation of $D^*_{f,g}(s;h;\delta)$ is given by the same expression. The analytic continuation of the integral requires more terms as you pass over the lines where the integrand has a pole for some value of the integrating parameter. 
		
		The following theorem is obtained in a very similar manner to the integral weight case. Notice that the information about the modular form $f$ is contained only in the inner products $\langle u_j,V\rangle$ and $\langle E_{\mathfrak{a}}(*,\hf + it)\rangle$. So the only difference of Theorem \ref{thm:AnalyticContinuation} and (5.5) in Proposition 5.1 of \cite{JeffShiftedSum} is the different regions of convergence.
		
		\begin{theorem}\label{thm:AnalyticContinuation}
			The function $D(s;h)$ given by \eqref{eq:shiftedDirichlet} has a meromorphic continuation to the entire complex plane and in the region $\Re(s)< -\tfrac14-k$, with $\Re(s)$ not a half integer, it is given by the formula
			\begin{align*}
				D^*_{f,g}&(s;h) \mspace{-2mu}=\mspace{-3mu} \sum_{j}\frac{\overline{\rho_j(-h)}\langle u_j , V\rangle}{(2\pi h)^{s-\hf}} \frac{\sqrt{\pi}}{2^{s-\hf}}\frac{\Gamma(s-\hf + it_j )\Gamma(s-\hf -it_j)\Gamma(1-s)}{\Gamma(\hf+it_j)\Gamma(\hf-it_j)}\\
				&+ \!\!\frac{1}{4\pi} \mspace{-3mu}\sum_{\mathfrak{a} \text{ cusp}} \!\int\limits_{(\hf)}\!\! \frac{\rho_{\mathfrak{a}}(1\mspace{-4mu}-\mspace{-4mu}u,\!-h)\langle E_{\mathfrak{a}}(*,\mspace{-4mu}u),V\rangle}{h^{s+u-1}}\frac{\pi^{2-s-u}}{2^{2s-2}}\frac{\Gamma(\!s\mspace{-3mu}+\mspace{-3mu}u \mspace{-3mu} -\mspace{-5mu}1 \!)\Gamma(\!s\mspace{-3mu}-\mspace{-3mu}u\!)\Gamma(\!1\mspace{-3mu}-\mspace{-3mu}s\!)}{\Gamma(u)\Gamma(1-u)^2} \d u\\
				&+\!\! \sum_{\mathfrak {a} \text{ cusp}} \!\!\sum_{r = 1}^{\lfloor \hf -  \sigma \rfloor}\mspace{-4mu} \frac{\overline{\rho_{\mathfrak{a}}(1\mspace{-3mu}-\mspace{-3mu}s\mspace{-3mu}-\mspace{-3mu}r,\mspace{-3mu}-\mspace{-2mu}h)}(-\!1)^r\Gamma(1\mspace{-3mu}-\mspace{-3mu}s) \Gamma(2s\mspace{-3mu}+\mspace{-3mu}r\mspace{-3mu}-\mspace{-4mu}1) }{h^{2s+r-1}\pi^{s+r-1}r!\Gamma(s\!+\!r)\Gamma(1\!-\!s\!-\!r)^2} \langle E_{\mathfrak{a}}(*,\mspace{-3mu}1\mspace{-3mu}-\mspace{-3mu}s\mspace{-3mu}-\mspace{-3mu}r), V\rangle.
			\end{align*}
		\end{theorem}
		
		Notice that there is a region of the complex plane where $D(s;h)$ does not have an explicit expression, namely the region where $-1/4 - k < \Re(s) <1$. For positive $\delta > 0$, $D_{f,g}(s;h;\delta)$ is a meromorphic function defined on the whole plane. The limit as $\delta$ goes to zero, has explicit convergent expressions on some left and right-half planes. In Section 5 of \cite{JeffShiftedSum}, a contour integral has been constructed, which gives the value of $D(s,h)$ as sums of contour integrals in these left and right half-planes.
		

	\subsection{Analytical continuation of $Z_Q(s,w)$}\label{subsec:ZQ}
		
		We obtain $Z_Q(s,w)$ by adding the shifted Dirichlet series as in Definition \ref{def:ZQ}. Supressing all but the $h$ dependence we have
		\[
			D(s;h) \ll h^{\frac{k-1}2}.
		\]
		Therefore this sum converges for $\Re(s), \Re(w)>1$. 
		
		Now we do the addition in the spectral domain, that is for $s \in \C$ with $\Re(s) < -1/4 - k$, do the summation for $\Re(w)$ large enough, so that the summation of the expression in Theorem \ref{thm:AnalyticContinuation} converges. First some definitions.   
		\[
			s' = s - \hf + w + \frac{k-1}{2},
		\]
		The summation will be over $h$ which are multiples of $Q$. The spectral expansion has the $h$\th\ Fourier coefficient of Maass forms and Eisenstein series, therefore after summing over $h$ we will obtain Dirichlet series, which resembles $L$-functions of $\overline{u_j}$ except that the Dirichlet series contains only multiples of $Q$:
		\begin{equation}\label{eq:LQ}
			L_Q(s,\overline{u_j}) : = \sum_{h = 1}^\infty \frac{\overline{\rho_j(-hQ)}}{(hQ)^s}.
		\end{equation}
		Similarly,
		\begin{equation}\label{eq:ZetaQ}
			\zeta_{Q,\mathfrak{a}}(s,u) := \zeta(2-2u)\sum_{h=1}^\infty \frac{\rho_{\mathfrak{a}}(1-u, -hQ)}{(hQ)^{s + u-\hf}}.
		\end{equation}
		This notation differs from that in \cite{JeffShiftedSum} by a factor of $Q^{-s}$. Finally for notational convenience let us create a notation for the gamma factors:
		\begin{equation}\label{eq:GammaFactors}
			G(s,u) := \hf(4\pi)^k\frac{\Gamma(s+u-1)\Gamma(s-u)\Gamma(1-s)}{\Gamma(u)\Gamma(1-u)\Gamma(s+k-1)}.
		\end{equation}
		Then one obtains, for $\Re (s') > 1$, the expression
		\[
			Z_Q(s,w) = Z_{Q,\text{cusp}}(s,w) + Z_{Q,\text{cts}}(s,w)
		\]
		where 
		\[
			Z_{Q,\text{cusp}}(s, w) =  \sum_{j} L_Q(s',\overline{u_j}) \langle u_j , V\rangle G(s, \tfrac12 + it_j)
		\]
		and
		\begin{align*}
			Z_{Q,\text{cts}}&(s,\mspace{-2mu}w) \mspace{-3mu}=  \frac{1}{2\pi} \!\sum_{\mathfrak{a} \text{ cusp}} \int_{(\hf)}\zeta_{Q,\mathfrak{a}}(s',u) \langle E_{\mathfrak{a}}^*(*,u),V\rangle  \frac{G(s,u)}{\zeta^*(2-2u)\zeta^*(2u)}  \d u\\
			+&\!\! \sum_{\mathfrak {a} \text{ cusp}} \mspace{-8mu}\sum_{r = 1}^{\lfloor \sigma - \hf\rfloor}\mspace{-6mu}\sum_{h=1}^\infty\mspace{-4mu} \frac{\overline{\rho_{\mathfrak{a}}(\mspace{-2mu}1\mspace{-4mu}-\mspace{-4mu}s\mspace{-4mu}-\mspace{-4mu}r,\mspace{-3mu}-\mspace{-2mu}h)}(-\!1)^r\Gamma(1\mspace{-4mu}-\mspace{-4mu}s) \Gamma(\mspace{-2mu}2s\mspace{-4mu} + \mspace{-4mu}r\mspace{-4mu} - \mspace{-4mu} 1 \mspace{-2mu}) }{h^{2s+r+w+(k-3)/2}r!\Gamma(s\!+\!r)\Gamma(1\!-\!s\!-\!r)^2} \frac{\langle E_{\mathfrak{a}}(*,\mspace{-3mu}1\mspace{-3mu}-\mspace{-3mu}s\mspace{-3mu}-\mspace{-3mu}r), \mspace{-2mu}V\rangle}{\Gamma(s+k-1)}.
		\end{align*}
		
		Let us denote the residue of $\zeta_{Q,\mathfrak a}$ as follows,
		\[
			\Res_{u = \hf \pm(1-s)} \zeta_{Q,\mathfrak a}(s,u) = K^{\pm}_{Q,\mathfrak{a}}(s) \zeta(2s-1)
		\] 
		where $K^{\pm}_{Q,\mathfrak{a}}(s)$ is some ratio of Dirichlet polynomials.
		
		\begin{theorem}\label{thm:ZQanalCont}
			With the notations above, the double Dirichlet series $Z_Q(s,w)$ can be meromorphically continued to $\Re(w)>1$, with simple polar lines at: $$s = \hf \pm it_j -r$$ for $r$ a nonnegative integer and $\tfrac14 + t_j^2$ eigenvalues of Maass forms on $\Gamma_0(N) \backslash \H$, $$s = \hf - r$$ for $r$ a nonnegative integer, $$w+ 2s + \frac k2 = \frac 52 -r $$ for a nonnegative integer $r$. 
			
			If $t_j=0$ for some $j$, i.e. if $\frac 14$ is an eigenvalue of $\Delta$ on $\Gamma_0(N)\backslash \H$, then the poles at $s = \hf \pm it_j - r$ will be double poles. 
			
			Furthermore there are poles at $$s = \frac \rho2 -r$$ where $\rho$ is  a nontrivial zero of $\zeta(s)$ with the same order as the pole of $\zeta(2s + 2r)\inv$. 
			  
			The residues at these poles are given by
			\[
				\Res_{s = \hf \pm it_j -r} Z_Q(s,w) = c_{r,j}L_Q(w + \tfrac{k-1}2 - r \pm it_j, \overline{u_j}),
			\]  
			where 
			\[
				c_{r,j} = \frac{(-1)^r}{r!}\frac{\pi^{\hf -it_j + r}}{2^{2it_j - 2r}} \frac{ \Gamma( 2it_j -r)\Gamma(\hf-it_j +r)}{\Gamma(\hf+it_j)\Gamma(\hf-it_j )\Gamma(k - r - \hf + it_j)}\langle u_j,V \rangle,
			\]
			and
			\begin{align*}
				&\Res_{w= \frac 52 -2s-\frac k2- r}\mspace{-23mu}Z_Q(s,w)\mspace{-3mu} =\mspace{-3mu} \frac{(4\pi)^k}{2\Gamma(s+k-1)}\sum_{\mathfrak{a}}\lt[\frac{\mathcal{V}(-1)^rG(s,s+r)}{r!\pi^{s+r-1}\zeta^*(2s+2r)\Gamma(1-s-r)}\rt.\\
				& \qquad\lt.\times\mspace{-3mu}\lt(K_{Q,\mathfrak a}^+(\tfrac 32 \mspace{-3mu}-\mspace{-3mu}s \mspace{-3mu}- \mspace{-3mu}r)\langle E_{\mathfrak{a}}^*(*,s\mspace{-3mu}+\mspace{-3mu}r),V \rangle\mspace{-5mu} -\mspace{-5mu} K_{Q,\mathfrak{a}}^- (\tfrac 32 \mspace{-3mu}-\mspace{-3mu} s\mspace{-3mu}-\mspace{-3mu} r)\langle E^*_{\mathfrak{a}}(*,1\mspace{-3mu}-\mspace{-3mu}s\mspace{-3mu}-\mspace{-3mu}r),V \rangle \rt)\rt]\mspace{-5mu}.
			\end{align*}
			The function also satisfies the following bounds. There are constants $A', A''$ depending only on $k$ such that for $\Re(s') \geq \hf$ and $\Re(w) >1$
			\begin{equation}\label{eq:ZQbound}
				Z_{Q,\text{cusp}}(s,w) \ll Q^{\theta - s'}L^{1-k+\vep} (1+|s'|)^{1+\vep}(1+|s|)^{3-k-3\sigma}
			\end{equation}
		\end{theorem}
		
		By Stirling's formula we can see that, 
		\[
			c_{r,j} \ll_{r,k} (1+|t_j|)^{\hf-k+r}\langle u_j, V\rangle.
		\]

		The proof of the theorem is almost verbatim that of Proposition 7.1 in \cite{JeffShiftedSum}. The proof is given using Hartog's theorem of analytic continuation of multivariable functions to convex domains. The original domains are $\Re(s), \Re(w)>1$ from the Dirichlet series, $\Re(s) < \qtr - k, \Re(w)>1$ from the spectral expansion, a region with sufficiently large real part for $w$ connecting these two regions. The only difference in the even weight modular forms is the region of convergence of the spectral expansion which is $\Re(s)< \hf- \frac k2, \Re(w)>1$. 
			
		In \cite{JeffShiftedSum}, the meromorphic continuation of $Z_Q(s,w)$ to all of $\C^2$ is obtained in the case of even integral weight. We will not do the necessary changes in the proof to obtain the same result, because in the next section we will only need the $w$ variable with real part greater than 1.
		

\section{Shifting lines of integration}\label{sec:movingLines}

	The goal of this section is to bound $S_2$ and $S_3$ as in Proposition \ref{prop:boundSWithShiftedSums}. We start with the triple inverse Mellin transform \eqref{eq:tripleMellin} and move the lines of integration. First move $s$ line of integration to $\Re(s) = -2\vep$. In doing that we pass over poles of $Z_Q$ and pick up residues. Estimating each, we obtain the following theorem.
	
	\begin{theorem}\label{S2S3bound}
		Let $H$ be a smooth cutoff function supported in the interval $[1,2]$. Consider the sums,
		\begin{align*}
			S_2 &= \sum_{\substack{m_1,m_2, h>0\\ m_1\ell_1 = m_2\ell_2 + hQ }}A(m_1)\overline{A(m_2)}H\lt(\frac{m_1}{X}\rt)H\lt(\frac{m_2}{X}\rt) \\
			S_3 &= \sum_{\substack{m_1,m_2, h>0\\ m_2\ell_2 = m_1\ell_1 + hQ }} A(m_1)\overline{A(m_2)}H\lt(\frac{m_1}{X}\rt)H\lt(\frac{m_2}{X}\rt)
		\end{align*}
		We can bound the shifted sums $S_2$ and $S_3$ by,
		\[
			S_2, S_3 \ll \frac{XL^{1+\vep}}{\sqrt{Q}},
		\]
		where $\vep>0$ is an arbitrarily small number, and the implied constant depends only on $k$, and the smooth cutoff function $H$.
	\end{theorem}
	
	\begin{proof}
		The sums $S_2$ and $S_3$ are symmetric, and in fact just complex conjugates of one another, so it is enough to bound only $S_2$. We will use \eqref{eq:tripleMellin} for dominating this shifted sum. Let us first deal with the cuspidal contribution. Define,
		\begin{align*}
			S_{2,\text{cusp}}^* &:= \lt(\frac{1}{2\pi i}\rt)^3 \int\limits_{(\frac{k+1}{2})}\int\limits_{(1+2\vep)} \int\limits_{(\frac{k+1}{2}+\vep)} Z_{Q,\text{cusp}}\lt(s+w-u,u-\tfrac{k-1}{2}\rt)\notag \\
			&\aquad \qquad\times\frac{\Gamma\lt(w+\frac{k-1}{2}-u\rt)\Gamma(u)}{\Gamma\lt(w+\frac{k-1}{2}\rt)} h(s)h(w) \ell_2^s \ell_1^w X^{s+w} \du \dw \ds \displaybreak[1]\\
			&= \lt(\frac{1}{2\pi i}\rt)^3 \int\limits_{(1+\vep)}\int\limits_{(1+2\vep)} \int\limits_{(1+\vep)} Z_{Q,\text{cusp}}\lt(s,u\rt) \frac{\Gamma\lt(w-u\rt)\Gamma(u+\tfrac{k-1}{2})}{\Gamma\lt(w+\frac{k-1}{2}\rt)} \\
			&\aquad\qquad\times h(s-w+u+\tfrac{k-1}{2})h(w)  \lt(\tfrac{\ell_1}{\ell_2}\rt)^w (\ell_2X)^{s+u+\frac{k-1}2} \du \dw \ds,
		\end{align*}
		where we have made a change of variables in the second line. Making the obvious definition for $Z_{Q,\text{cts}}$ and noting the difference in \eqref{eq:twoDifferentZs} we realize that
		\[
			S_2 = (\ell_1\ell_2)^{\frac{k-1}{2}}\lt(S_{2,\text{cusp}}^* + S_{2,\text{cts}}^*\rt)
		\]
		Now move the new $s$ line of integration all the way to $\Re(s) = -k- \qtr - \vep$, in the way passing over the poles at $s = \hf + it_j - r$ for all $r = 0,1, \ldots, k - \tfrac 12$. 
		\begin{align}
			S_{2,\text{cusp}}^* =& \lt(\frac 1{2\pi i }\rt)^3 \int\limits_{(-k - \qtr - \vep)}\int\limits_{(1+2\vep)} \int\limits_{(1+\vep)} Z_Q\lt(s,u\rt) \frac{\Gamma\lt(w-u\rt)\Gamma(u+\tfrac{k-1}{2})}{\Gamma\lt(w+\frac{k-1}{2}\rt)}\notag \\[-1ex]
			&\quad\times h(s-w+u+\tfrac{k-1}{2})h(w)  \lt(\tfrac{\ell_1}{\ell_2}\rt)^w (\ell_2X)^{s+u+\frac{k-1}2} \du \dw \ds\label{eq:S2cuspidalRemainder}\\
			&+ \sum_{r=0}^{k-\hf}\sum_j\lt(\frac{1}{2\pi i}\rt)^2\mspace{-10mu} \int\limits_{(1+2\vep)} \int\limits_{(1+\vep)}  c_{r,j}L_Q(u + \tfrac{k-1}2 - r + it_j, \overline{u_j})\Gamma\lt(w-u\rt) \notag\\[-1ex]
			&\quad \times\frac{\Gamma(u+\tfrac{k-1}{2})}{\Gamma\lt(w\mspace{-3mu}+\mspace{-3mu}\frac{k-1}{2}\rt)} h(u\mspace{-3mu} - \mspace{-3mu} w \mspace{-3mu}+ \mspace{-3mu} \tfrac{k}{2} \mspace{-3mu}+\mspace{-3mu} it_j)h(w)  \mspace{-5mu} \lt(\tfrac{\ell_1}{\ell_2}\rt)^w \mspace{-7mu} (\ell_2X)^{u-r+\frac{k}2 + it_j} \du \dw.\label{eq:S2cuspidalResidue}
		\end{align}
		We will bound both the summands \eqref{eq:S2cuspidalRemainder} and \eqref{eq:S2cuspidalResidue}, which we call $I$ and $R$ respectively. In the $u$ integral of $R$ shift the line of integration so that $\Re(u + (k-1)/2 - r + it_j) = 1/2$. Notice that the power of $(\ell_2X)$ has real part equal to $1$. Furthermore the gamma factors, and the rapidly decaying functions $h$ ensure the convergence of the integrals in the $u$ and the $w$ variables and also the summation over $j$. We only need to bound the following quantity, 
		\[
			C(T,s) = \sum_{T\leq |t_j| \leq 2T} c_{r,j}L_Q(s, \overline{u_j}),
		\]
		with a reasonable growth in the imaginary part of $s$ and in $T$, and explicit growth in the $L$ and the $Q$ parameters. Now apply Cauchy-Schwartz inequality.
		\begin{align*}
			C(T,s) &\leq T^{\hf - k + r}\lt(\mspace{-3mu}\sum_{T\leq |t_j| \leq 2T} \mspace{-7mu}\lt|L_Q(s, \overline{u_j})\rt|^2 e^{-\pi|t_j|} \mspace{-5mu}\rt)^{\mspace{-5mu}\hf}\mspace{-5mu} \lt(\sum_{T\leq |t_j| \leq 2T} \mspace{-7mu}\lt|\langle u_j, V\rangle \rt|^2e^{\pi |t_j|} \rt)^{\mspace{-5mu}\hf}\\
			&\ll T^{\hf - k + r} Q^{\theta - s} L^{1+\vep }(1+|s|+T)^{1+\vep}L^{-k}T^{2k+\frac 32} 	\\
			&\ll T^{k+r + 3}(1+|s|)^{1+\vep}Q^{\theta - s}L^{1-k + \vep}, 
		\end{align*}
		where for the first inequality we have used (7.15) of \cite{JeffShiftedSum} and for the second one, Corollary \ref{cor:myMainContribution}. 
		
		We get the bound,
		\begin{equation}
			R \ll \frac{L^{1-k}(LX)^{1+\vep}}{\sqrt{Q}}.
		\end{equation}
		
		We now bound \eqref{eq:S2cuspidalRemainder}. Move the line of integration in $w$ past that of $u$, picking up a pole at $w = u$, and then move the $u$ line of integration to the right, so that $\Re(s') = \Re(s + u + k/2-1) = 1/2$, in both the integral and the residue. The integral can be bounded:
		\begin{align*}
			I =& \lt(\frac{1}{2\pi i}\rt)^2\mspace{-14mu}\int\limits_{(-k - \qtr - \vep)}\int\limits_{(\frac{k}{2}+\frac 74 + \vep)} \mspace{-5mu} Z_Q(s,u) h(s \mspace{-3mu}+\mspace{-3mu}  \tfrac{k-1}2) h(u) \lt(\frac{\ell_1}{\ell_2}\rt)^{\mspace{-4mu}u}\mspace{-4mu} \lt(\ell_2X\rt)^{s+u + \frac {k-1}2} \d u \d s\\
			&+ \lt(\frac{1}{2\pi i}\rt)^3 \int_{(-k - \qtr - \vep)}\int_{(1+\frac{\vep}{2})} \int_{(\frac{k}{2}+\frac 74 + \vep)} Z_Q\lt(s,u\rt) \frac{\Gamma\lt(w-u\rt)\Gamma(u+\tfrac{k-1}{2})}{\Gamma\lt(w+\frac{k-1}{2}\rt)} \\
			&\aquad\qquad\times h(s-w+u+\tfrac{k-1}{2})h(w)  \lt(\tfrac{\ell_1}{\ell_2}\rt)^w (\ell_2X)^{s+u+\frac{k-1}2} \du \dw \ds.
		\end{align*} 
		The integrals converge due to the fast vanishing of the $h$ and the gamma functions vertically, and we can use the bound \eqref{eq:ZQbound} in Theorem \ref{thm:ZQanalCont} to obtain the same bound for $I$ as for $R$. This gives that
		\[
			S_{2,\text{cusp}}^* \ll \frac{L^{1-k+\vep}(LX)^{1+\vep}Q^{\theta}}{\sqrt{Q}}  
		\]
		
		The bound for $S_{2,\text{cts}}^*$ is proven the same way as in \cite{JeffShiftedSum}, no new input such as Corollary \ref{cor:myMainContribution} is required as was the case in the discrete spectrum. We therefore refer the reader there. Combining it all together we can bound the shifted sum $S_2$ (and similarly also $S_3$),
		\begin{equation*}
			S_2,S_3\ll \frac{(XL)^{1+\vep}Q^{\theta}}{\sqrt{Q}}.
		\end{equation*}
	\end{proof}


\section{Automorphic Kernel} \label{sec:AutomorphicKernel}

	The rest of the paper is concerned with the estimation of the quantity 
	\[
		\langle f(\ell_1z)\overline{f(\ell_2z)} y^k, u_j(z)\rangle,
	\]
	as a function of the $\ell$'s and the eigenvalue of the Maass form $u_j$ which equals $\frac 14 + t_j^2$. We will realize the inner product, more precisely some smoothed sum over the inner products as a pairing of $V(z) = f(\ell_1z) \overline{f(\ell_2z)}y^{k}$ with an automorphic kernel. Let $k(z,z')= k(u(z,z'))$ be a point pair invariant, that is let it only depend on the hyperbolic distance between the two points $z,z' \in \H$. Here $$u(z,z') = \frac{|z-z'|^2}{4yy'}$$ is a convenient parameter for hyperbolic distance, where the actual distance is related to $u$ via the equation, $$\cosh d_\H(z,z') = 2u(z,z') + 1 .$$ Let $\Gamma = \Gamma_0(N)$, we define the automorphic kernel related to this point pair invariant, and to $\Gamma$ as $$K(z,z') = \sum_{\gamma \in \Gamma} k(z,\gamma z').$$ 
	
	The function $K$ has some remarkable spectral properties, if $F$ is any eigenfunction of the hyperbolic Laplacian, and is invariant under $\gamma \in \Gamma$, then the integral operator defined by $K$ has $F$ as an eigenfunction. More precisely, if $\Delta F (z) = \lt(\frac 14 + t^2\rt)F(z)$, then
	\[
		\frac1{\Vol (\Gamma\backslash \H)}\iint_{\Gamma \backslash\H} K(z,z')F(z) \dz  = h(t)F(z).
	\]
	for some function $h$, determined only by the function $k$. See \cite{IwaniecSpectralBook} section 1.8 for details. In fact one may obtain $h$, given $k$ using a three step transformation called the Harish-Chandra--Selberg transform. This transform can also be seen as the Fourier transform on the hyperbolic plane, restricted to radial functions, see \cite{HelgasonGGA}. The exact form of the transform follows the steps,
	\begin{align*}
		q(v) &= \frac1 \pi \int_v^\infty \frac {k(u) \du}{\sqrt{u-v}} = \frac1\pi \int_{-\infty}^\infty k(v+w^2)\dw,\\
		g(r) &= 2 q\lt(\sinh^2\lt(\frac r2\rt)\rt),\\
		h(t) &=  \int_{-\infty}^\infty g(r) e^{irt}\dr. 
	\end{align*}
	One can also obtain this transformation with a single step:
	\[
		h(t) = \int_1^\infty k(u) P_{-\frac 12+it}(u) \du,
	\]
	where $P_{\nu}(u)$ is the Legendre $P$ function of order $\nu$. Furthermore, this transformation may be inverted by the following two (equivalent) sets of equations. Either one performs the inversion of the three steps one by one, first inverting the Fourier transform on the real line, then the change of variables and then the Abel transform,
	\begin{align*}
		g(\xi) &= \frac{1}{2\pi} \int_{-\infty}^\infty h(t)e^{-it\xi} \dt,\\
		q\lt(\frac{e^\xi + e^{-\xi} -2}4\rt) &= \frac 12 g(\xi),\\
		k(u) &= -\frac{1}{\pi} \int_v^\infty \frac{q'(v)\dv}{\sqrt{v-u}};
	\end{align*} 
	or one just inverts the single step transform using the Legendre $P$ function,
	\[
		k(u) = \frac1{2\pi}\int_{-\infty}^\infty P_{-\hf+it}(\cosh u) h(t) t\tanh (\pi t) \dt.
	\]
	
	The spectral expansion of the automorphic kernel can be obtained via its property against eigenfunctions of the hyperbolic Laplacian,
	\begin{equation*}
		K(z,z') = \sum_{t_j} h(t_j) u_j(z)\overline{u_j(z')} + \frac{1}{4\pi}\sum_{\mathfrak{a} \text{ cusp}} \int_{-\infty}^{\infty} h(t) E_{\mathfrak{a}}(z,\tfrac12+it)E_{\mathfrak{a}}(z,\tfrac12 - it)\dt.
	\end{equation*}
	Here the first sum is over all the Maass forms of level $N$, and the sum before the integral is over cusps of $\Gamma \backslash \H$. Furthermore, $E_\mathfrak{a}$ denotes the Eisenstein series of level $N$, associated to the cusp $\mathfrak{a}$.

	We can obtain a smoothed sum of inner products by pairing the automorphic kernel with the $\Gamma$-invariant function $f_1(z) \overline{f_2(z)} y^k$. Using the spectral decomposition of $K$ as above, and calling $\mathcal{V} = \Vol(\Gamma\backslash \H)$,
	\begin{align}
		\langle f_1(z)&\overline{f_2(z)}y^k, K(z,z')\rangle = \sum_{t_j}\frac{1}{\mathcal V} \iint_{\Gamma\backslash \H}f_1(z)\overline{f_2(z)}y^k \overline{h(t_j)}\overline{u_j(z)} u_j(z') \dz \notag\\
		&+ \frac1{4\pi} \sum_\mathfrak{a} \frac{1}{\mathcal{V}} \iint\limits_{\Gamma\backslash \H} f_1(z)\overline{f_2(z)}y^k\int\limits_{-\infty}^\infty h(t) E_\mathfrak{a}(z,\tfrac12 -it)E_\mathfrak{a}(z',\tfrac 12 + it) \dt \dz\notag\\
		&= \sum_{t_j} h(t_j) \langle f_1(z)\overline{f_2(z)}y^k, u_j(z)\rangle u_j(z') \notag\\
		&+ \frac{1}{4\pi} \sum_\mathfrak{a} \int\limits_{-\infty}^{\infty} h(t) \langle f_1(z)\overline{f_2(z)}y^k, E_{\mathfrak{a}}(z,\tfrac12+it)\rangle E_{\mathfrak{a}}(z',\tfrac12+it)\dt \notag.
	\end{align}
	
	Now we take the $L^2$ norm of both sides with respect to the $z' \in \Gamma\backslash \H$ variable. Using Plancharel's theorem we have,
	\begin{align}
		\sum_{t_j} &|h(t_j)|^2 |\langle f_1(z)\overline{f_2(z)}y^k, u_j(z)\rangle|^2\label{eq:tripleProdSmooth}\\
		&\leq \sum_{t_j} |h(t_j)|^2 |\langle f_1\overline{f_2}y^k, u_j(z)\rangle|^2\mspace{-4mu}+\mspace{-4mu} \frac{\mathcal{V}}{4\pi }\sum_{\mathfrak a} \int\limits_{-\infty}^\infty \mspace{-4mu} |h(t)|^2|\langle f_1\overline{f_2}y^k E_{\mathfrak{a}}(*,\tfrac12\mspace{-3mu}+\mspace{-3mu}it)\rangle|^2 \dt \notag\\
		&= \| \mathcal{V}\langle f_1(z)\overline{f_2(z)}y^k, K(z,z')\rangle_{z}\|^2_{L^2(z'\in \Gamma\backslash \H)}\notag\\
		&\leq \lt(\sup_{z' \in \H} \lt|\mathcal{V} \langle f_1(z)\overline{f_2(z)}y^k, K(z,z')\rangle_{z}\rt|\rt)^2. \notag
	\end{align}
	
	We will use the function $h$ in the sum $\eqref{eq:tripleProdSmooth}$ as a localizing bump function which will concentrate the sum near some given magnitude $T$, and then estimating the size of the inner products will be a matter of bounding
	\begin{equation}
		\mathcal{V}\langle f_1(z)\overline{f_2(z)}y^k, K(z,z')\rangle_{z} \label{Boundee}.
	\end{equation}


\section{Hyperbolic Disc Model}\label{sec:hyperbolicDisc}

	Let us pass to the disc model of the hyperbolic plane which we will denote by $\D$, with $z'$ sent to the center, via the new coordinates
	\[
		w = \frac{z -z'}{z-\overline {z'}}.
	\]
		Let $\phi_1, \phi_2$ be $f_1$ and $f_2$ expressed in the $w$ coordinates, i.e.
	\[
		\phi_i (w) = f_i(z) = f_i\lt(\frac{-\overline {z'}w + z'}{-w+1}\rt) \qquad (i=1,2).
	\]
		Let $w = \rho e^{i\theta}$ be polar coordinates, and we will denote by $r(w)$ the hyperbolic distance from origin to the point $w$, in this case $$\rho= \tanh\lt(\frac r2\rt).$$ A calculation shows that if $y'$ denotes the imaginary part of the point $z'$, the function $y^s$ on the upper half plane will transform to the function
	\[
		\lt(\frac{1-|w|^2}{|1-w|^2}\rt)^sy'^s.
	\]
	
	We unfold the integral in \eqref{Boundee} and pass to the hyperbolic disc.
	\begin{align}
		\mathcal V\langle &f_1(z) \overline{f_2(z)}y^k, K(z,z')\rangle = \iint_\H f_1(z) \overline{f_2(z)}y^k k(z,z') \dz\notag\\
		&=\iint_\D \phi_1(w)\overline{\phi_2(w)} \lt(\frac{1-|w|^2}{|1-w|^2}\rt)^k y'^k k\lt(\frac{|w|^2}{1-|w|^2}\rt) 4\frac{\mathrm{d}w_x \mathrm{d}w_y}{(1-|w|^2)^2}\notag\\
		&= \frac 1{2\pi} \int_0^1 \int_0^{2\pi} \phi_1(\rho e^{i\theta}) \overline{\phi_2(\rho e^{i\theta})}\lt(\frac {1-\rho^2}{|1-\rho e^{i\theta}|^2}\rt)^k y'^k k\lt(\frac{\rho^2}{1-\rho^2}\rt) \frac{4\rho \mathrm d \rho}{(1-\rho^2)^2} \dth. \label{eq:discIntegral}
	\end{align}
	Let us first integrate over $\theta$, for that we use the Taylor expansions of the holomorphic functions $\phi_1$ and $\phi_2$,
	\[
		\phi_1(w) = \sum_{n =0}^\infty a_nw^n, \qquad \phi_2(w) = \sum_{m = 0}^\infty b_m w^m.
	\]
	Then the integral with respect to $\theta$ in \eqref{eq:discIntegral} gives us,
	\begin{align}
		B(\rho):=\frac{1}{2\pi}\int_{0}^{2\pi}& \lt(\sum_{n=0}^\infty a_n \rho^n e^{in\theta}\rt) \lt(\sum_{m=0}^\infty \overline{b_m}\rho^m e^{-im\theta}\rt) \lt(\frac{1-\rho^2}{|1-\rho e^{i\theta}|^2}\rt)^k \dth\notag\\
		&=\sum_{n,m=0}^\infty a_n \overline{b_m} \rho^{n+m} \frac{1}{2\pi}\int_0^{2\pi} e^{i(n-m)\theta} \lt(\frac{1-\rho^2}{|1-\rho e^{i\theta}|^2}\rt)^k \dth. \label{eq:powerPoisson}
	\end{align}
	The inner integral can be evaluated exactly, and as can be seen from \cite{ErdelyiHighTransFuncI}, on page 81, equation 10,
	\[
		\frac{1}{2\pi} \int_0^{2\pi} \lt(\frac{1-\rho^2}{|1-\rho e^{i\theta}|^2}\rt)^k \dth = (1-\rho^2)^k \rho^{|h|} \frac{\Gamma(k+h)}{\Gamma(k)\Gamma(|h|)} F(k,k+|h|;|h|+1; \rho^2)
	\]
	with $F = {}_2 F_1$ being the Gauss hypergeometric function. Then we can evaluate our theta integral \eqref{eq:powerPoisson} to be, 
	\begin{align}
		B(\rho) &= \underbrace{\sum_{h=0}^\infty \sum_{n=0}^\infty a_{n+h}\overline{b_n} \rho^{2n+2h}  \frac{\Gamma(h+k)}{\Gamma(h+1)\Gamma(k)} (1-\rho^2)^k F(k,h+k;h+1;\rho^2)}_{B_1}\notag\\
		&+\underbrace{\sum_{h=1}^\infty \sum_{n=0}^\infty a_{n}\overline{b_{n+h}} \rho^{2n+2h}  \frac{\Gamma(h+k)}{\Gamma(h+1)\Gamma(k)} (1-\rho^2)^k F(k,h+k;h+1;\rho^2)}_{B_2}.\label{eq:hypergeometricShiftedSum}
	\end{align}
	
	In this calculation $\rho$ is taken to be a real variable, it is the distance from the origin in the polar coordinates. In what follows, however, it will be useful to take $\rho$ as a complex variable in order to estimate the integral,
	\begin{equation}\label{eq:rhoIntegral}
		\int_0^1 B(\rho) y'^k k\lt(\frac{\rho^2}{1-\rho^2}\rt) \frac{4\rho\mathrm{d}\rho}{(1-\rho^2)^2}.
	\end{equation}
	
	Expression \eqref{eq:hypergeometricShiftedSum} allows us to extend $B$ to a holomorphic function of $\rho$ in the unit disc. Furthermore, we can give a bound to it, which will be essential in further considerations. First of all the hypergeometric function is bounded on the unit disc uniformly for all $h$ via
	\[
		|F(k,h+k;h+1;\rho^2) (1-\rho^2)^{2k-1}| \leq 2^k. 
	\]
	This can be seen fist by noting that upon multiplying with $(1-\rho^2)^{2k-1}$ the singularity of the hypergeometric function at $\rho = 1$ is mended, see for example the linear transformation formula 15.3.6 in \cite{StegunHandbook}. Secondly Taylor series expansion of the hypergeometric function at $\rho = 0$ has as a limit, as $h \to \infty$,
	\[
		F(k,h+k;h+1,\rho^2) = \sum_{n=0}^\infty \frac{(k)_n(h+k)_n}{(h+1)_n n!} \rho^{2n} \longrightarrow \sum_{n=0}^\infty \frac{(k)_n}{n!} \rho^{2n}= \frac{1}{(1-\rho^2)^k}.
	\]
		
	Since $f_1$ and $f_2$ are cuspforms of weight $k$, the functions $|f_1|y^{\frac k2}, |f_2|y^{\frac k2}$ is a bounded function in the upper half plane, say by the values, $M_1$ and $M_2$ respectively. Then since the $\phi$'s are just a transfer of these functions to the hyperbolic disc model, we have that,
	\[
		|\phi_i(w)| y'^{\frac k2}\lt(\frac{1-|w|}{|1-w|^2}\rt)^{\frac{k}{2}} \leq M_i \qquad (i = 1,2).
	\]
	This can be used to get an upper bound on the sum appearing in $S_1$.
	\begin{align*}
		\lt|\sum_{n=0}^\infty a_{n+h} \overline{b_n} \rho^{2n+h} \rt|&= \lt|\frac 1{2\pi}\int_{0}^{2\pi} \phi_1(\rho e^{i\theta}) \overline{\phi_2(\overline{\rho}e^{i\theta})}e^{-ih\theta} \dth\rt|\\
		&\leq \frac{M_1M_2}{y'^k} \lt(\frac{1}{1-|\rho|^2}\rt)^k \frac 1{2\pi }\int_0^{2\pi}|1-|\rho|e^{i\theta}|^{2k}\dth\\
		&\leq \frac{M_1M_2}{y'^k} \frac{2^{2k}}{(1-|\rho|^2)^k}.
	\end{align*}
	Note that in the above calculation $\rho$ does not need to be real, in fact we took care in ensuring that the inequality remains valid for all $\rho$ in the unit disc. Now using this and the bound for the hypergeometric function we bound $B(\rho)$ in the unit disc,
	\begin{align*}
		B_1&\leq \sum_{h = 0}^\infty \lt|\sum_{n=0}^\infty a_{n+h}\overline{b_n} \rho^{2n+h}\rt| |\rho|^h \frac{\Gamma(h+k)}{\Gamma(k)h!}\frac{2^k}{(1-\rho^2)^{2k-1}}\\
		&\leq \frac{M_1M_2}{y'^k} \frac{2^{2k}}{(1-|\rho|^2)^k} \frac{2^k}{|1-\rho^2|^{2k-1}}\sum_{h=0}^\infty |\rho|^h \frac{\Gamma(k+h)}{\Gamma(k)h!}\\
		&\leq \frac{M_1M_2}{y'^k} \frac{2^{4k}}{(1-|\rho|^2)^{2k}}\frac{1}{|1-\rho^2|^{2k-1}}.
	\end{align*}•
	The sum $B_2$ can be bounded by the same quantity, and hence we obtain a bound for $B(\rho)$.
	

\section{The Harish-Chandra Selberg Transform}\label{sec:selbergTransform}

	Let us choose a localizing function $h$ which has an easy-to-calculate Fourier transform, since the first step in the three-step transform is a standard Fourier transform. We choose, as in \cite{SarnakIMRN94}, $$h(t) = h_{T}(t) = e^{-\pi(t-T)^2} + e^{-\pi(t+T)^2}.$$ Then its Fourier Transform is given by,
	\begin{equation}
		g(\xi) = 2\cos(\xi T) e^{-\pi\xi^2}\notag.
	\end{equation}
	The function $q$, obtained via a change of variables, satisfies
	\[
		q\lt(\frac{e^{\xi} + e^{-\xi}-2}{4}\rt) = \frac 12 g(\xi) = \cos(\xi T)e^{-\pi \xi^2}.
	\]
	After differentiating both sides with respect to $\xi$, we obtain,
	\[
		q'\lt(\frac{e^{\xi} + e^{-\xi} -2}{4}\rt) \frac{e^{\xi} - e^{-\xi}}{4} = -\lt(T\sin(\xi T) + 2\pi\xi \cos(\xi T) \rt)e^{-\pi \xi^2 }.
	\]
	We apply the above choice of localizing function $h$ to the integral \eqref{eq:rhoIntegral}. We first make a change of variables, changing the Euclidean distance $\rho$ with the hyperbolic distance $\tanh (r/2)$ from the origin. 
	\begin{align}
		\int_0^1 &B(r) y'^k k\lt(\frac{\rho^2}{1-\rho^2}\rt) \frac{4\rho\mathrm{d}\rho}{(1-\rho^2)^2} = \int_{0}^{\infty} B(\tanh r/2) y'^k k\lt(\sinh^2 \frac{r}{2}\rt) \sinh r \dr \notag\\
		&= \int_0^\infty B(\tanh (r/2)) y'^k \frac{1}{\pi}\int_r^\infty \frac{\lt(T \sin(\xi T) + 2\pi \xi \cos(\xi T) \rt) e^{-\pi\xi^2}  }{\sqrt{\sinh^2 \frac{\xi}{2} - \sinh^2 \frac{r}{2}}} \sinh r \dr \notag \\
		&= \frac{y'^k\sqrt{2}}{\pi} \int_0^\infty \underbrace{\int_0^\xi  \frac{B(\tanh r/2)\sinh r}{\sqrt{\cosh \xi - \cosh r}} \dr}_{H(\xi)} \lt(T\sin(\xi T) + 2\pi\xi \cos(\xi T)\rt)e^{-\pi \xi^2} \dxi \label{eq:infiniteIntegral}
	\end{align}
	The function $H(\xi)$ thus defined, can be extended to a holomorphic function of $\xi$ in the region $|\Im(\xi) | < \frac{\pi }{2}$, as the following formulation makes clear.
	\[
		H(\xi) = \int_0^1 \frac{B\lt(\tanh \frac{\xi \eta}{2}\rt) \sinh(\xi \eta) \eta}{\sqrt{\cosh \xi - \cosh \xi \eta}} \mathrm{d} \eta = \int_0^1 \frac{B\lt(\tanh \frac{\xi\eta}{2}\rt) \sinh \xi \eta}{\sqrt{ \frac{\sinh \xi(\frac{1+\eta}{2})}{\xi} \frac {\sinh \xi\lt(\frac{1-\eta}{2}\rt)}{\xi}}} \mathrm{d} \eta.
	\]
	Furthermore, from \eqref{eq:hypergeometricShiftedSum} it is apparent that $B(\rho)$ is an even function of $\rho$, and therefore, this integral shows $H(\xi)$ to be an odd function of $\xi$, thus making the integrand in the infinite integral \eqref{eq:infiniteIntegral} even. Our plan is to complete the integral to an integral on the whole real line, express $\cos$ and $\sin$ as linear combinations of exponential functions and move lines of integration up or down to give a bound for \eqref{eq:infiniteIntegral}. 
	
	In the region $|\Im(\xi)| < \frac{\pi}{2}$ we have $$\frac{\sinh \xi \lt(\frac{1\pm+\eta}{2}\rt)}{\xi} \geq  \frac{1\pm \eta}{2},$$ and $$\sinh \xi \eta \leq |\sinh \xi| \leq  e^{\Re(\xi)}, $$ 
	\[
		|B(\tanh \frac{\xi \eta}{2})| \leq \frac{M_1M_2}{y'^k} \lt(\frac{2}{1-|\tanh \xi \eta/2|^2}\rt)^{2k} \frac{1}{|1-\tanh^2 \xi \eta/2|}
	\] 
	The right hand side is an increasing function of $\eta$. So it is enough to consider the bound when $\eta = 1$. This allows you to write, for $\xi = x \pm i(\frac{\pi}2 -\frac 1T),$
	\[
		|B(\tanh \xi\eta/2)| \leq 2^{3k} \cosh^{2k} (x) T^{2k} \cosh^{2k-2} \lt(\tfrac x2\rt), 
	\]
	thus giving us on these two horizontal lines,
	\begin{align*}
		|H(\xi)| &\leq 2^{3k} \frac{M_1M_2}{y'^k}\cosh^{3k-1}x T^{2k} \int_0^1 \frac{1}{\sqrt{1-\eta^2}}\mathrm d \eta\\
		&\leq \pi \frac{M_1M_2}{y'^k} e^{(3k-1)x} T^{2k}.
	\end{align*}
	We bound the integral \eqref{eq:rhoIntegral} in order to give an upper bound for \eqref{Boundee}. After expressing $\cos x= (e^{ix}+ e^{-ix})/2$ we move the lines of integration to $\Im \xi = \pm (\frac \pi 2 -\frac 1T)$
	\begin{align*}
		\frac{y'^k}{\sqrt{2}\pi} \int_{-\infty}^\infty H(\xi) &(T\sin(\xi T) + 2\pi \xi \cos(\xi T) )e^{-\pi \xi^2}\dxi \\
		=& \frac{y'^k }{\sqrt 2 \pi} \int_{-\infty+i(\frac \pi2 - \frac 1T)}^{\infty+i\lt(\frac \pi2 - \frac 1T\rt)} H(\xi ) \lt(\frac {T}{2i} + \pi \xi\rt) e^{i\xi T} e^{-\pi \xi^2}\dx\\
		&+\frac{y'^k }{\sqrt 2 \pi} \int_{-\infty-i(\frac \pi2 - \frac 1T)}^{\infty-i\lt(\frac \pi2 - \frac 1T\rt)} H(\xi ) \lt(-\frac {T}{2i} + \pi \xi\rt) e^{-i\xi T} e^{-\pi \xi^2}\dx\\
		\leq & M_1M_2 T^{2k+1} e^{-\frac \pi2 T}e^{\frac{\pi^2}4+1}\int_{-\infty}^\infty e^{\frac{(3k-1)x}{\sqrt{\pi}}} \lt(1+\frac{\sqrt{\pi}|x|}{T}\rt) e^{- x^2} \dx\\
		\leq & 64 M_1M_2 T^{2k+1} e^{-\frac \pi2 T}e^{\frac{(3k-1)^2}{4\pi}}\lt(1+ \frac kT\rt) .
	\end{align*} 
	
	Summarizing everything up to now, we have proven the following inequality, using equations \eqref{eq:discIntegral}, \eqref{eq:rhoIntegral}, \eqref{eq:infiniteIntegral} this last piece of computation above.
	\begin{prop}\label{prop:kernelProdBound}
		Let $f_1$ and $f_2$ be half integral weight holomorphic modular forms of weight $k$ and level $N$. Let $\mathcal{V} = \Vol(\Gamma_0(N) \backslash \H)$ and let $M_i = \sup_{z \in \H}|f_i(z) y^{\frac{k}{2}} |$ for $i = 1,2$. Then one has the bound,
		\[
			\mathcal V\langle f_1(z) \overline{f_2(z)}y^k, K(z,z')\rangle\leq 64 M_1M_2 T^{2k+1} e^{-\frac \pi2 T}e^{\frac{(3k-1)^2}{4\pi}}\lt(1+ \frac kT\rt).
		\]
		Recall that the  quantity $T$ is encoded in the definition of $K(z,z')$.  
	\end{prop}
	
	Using the inequality in Proposition \eqref{prop:kernelProdBound} and equation \eqref{eq:tripleProdSmooth} we obtain Proposition \ref{prop:tripleInnerBound}. 
	
	Note that the equation also has a continuous spectrum. Although we will not need it in this paper, let us record that result as well.
	\begin{prop}\label{prop:tripleInnerBoundEisenstein}
		Given $f$ and $g$ two half integral weight modular forms of weight $k$,
			\[
				\int_{T}^{T+1} \mathcal{V} \lt|\langle f\overline{g}y^k,E_{\mathfrak{a}}(*,\tfrac12 + it) \rangle \rt|^2  \d t \ll \|fy^{\frac k2}\|_{L^\infty}^2 \|gy^{\frac k2}\|_{L^\infty}^2  T^{4k+2}e^{-\pi T}.
			\]
			where $E_{\mathfrak{a}}$ is the Eisenstein series of level $N$, weight $0$ at the cusp $\mathfrak{a}$. 
	\end{prop}
	
	Now we assume that $f$ is a fixed half integral modular form of weight $k$ and level $N_0$, and $\ell_1, \ell_2$ two primes. Let $$M = \sup_{z \in \H} |f(z)|y^\frac k2.$$ Then since we are only mapping the domain to itself, we also have
	\[
		M = \sup_{z \in \H} |f(\ell_i z)| \lt(\ell_i y\rt)^{\frac k2},
	\]
	thus if $f_i(z) = f(\ell_iz)$, then $M_i = M/\ell_i^{\frac k2}$. So we can bound the inner product of the product of $f_1$ and $f_2$ with the automorphic kernel. 
	\[
		\mathcal{V} \langle f(\ell_1z)\overline{f(\ell_2z)}y^k , K(z,z')\rangle \ll \frac{M^2}{(\ell_1\ell_2)^{\frac k2}} T^{2k+1} e^{-\frac{\pi}{2}T}e^{k^2} \lt(1+\frac kT\rt)
	\]
	Then using \eqref{eq:tripleProdSmooth} we can get the following bound.
	\begin{cor}\label{cor:myMainContribution}
		Preserving all the notations as above,
		\begin{equation*}	
			\sum_{T<|t_j|< 2T} |\langle f(\ell_1z)\overline{f(\ell_2z)}y^k, u_j(z) \rangle|^2 e^{-\pi |t_j|} \ll_f  T^{4k+3} (\ell_1\ell_2)^{-k}.
		\end{equation*}
	\end{cor}
	
	Hence we have filled one crucial step in the proof of \eqref{subconvexitybound}, the subconvexity bound for the $L(\hf,f, \chi)$ in terms of the conductor of $\chi$ the twisting character.


\section{Subconvexity}\label{sec:subconvexity}
	Now we combine everything we have proven to get at a subconvex bound for the $L$ function of a half integer weight modular form, twisted by characters modulo $Q$. First by Section \ref{sec:amplification} equation \eqref{eq:shortsum}, we get that,
	\begin{equation*}
		L\lt(\hf,f,\chi \rt) \ll (\sqrt{N}Q)^{-\hf} \max_{x\ll (\sqrt{N}Q)^{1+\vep}} \lt|\sum_m A(m) \chi(m) H\lt(\frac mx\rt)\rt|.
	\end{equation*}
	To get a bound for the sum within, we make use of the $S$ in \eqref{eq:amplifiedSum}. Ignoring all but the $\psi = \chi$ term we get that,
	\[
		\lt|\sum_m A(m) \chi(m) H\lt(\frac mx\rt)\rt| \leq \frac{S^\hf\log L}{L}.
	\]
	Finally from the inequality \eqref{eq:boundSWithShiftedSums} and Theorems \ref{thm:diagonal} and \ref{S2S3bound} we can say,
	\begin{align*}
		S &\leq \phi(Q) \sum_{\substack{\ell_1, \ell_2 \text{ primes}\\ (\ell_1\ell_2, NQ)=1}} \chi(\ell_1)\overline{\chi(\ell_2)}(S_1+S_2+S_3) \\
		&\ll Q (L X + \sqrt{X}L^3 + \frac{XL^{3+\vep}Q^{\theta}}{\sqrt{Q}}).
	\end{align*}
	So we can bound
	\begin{align*}
		L\lt(\hf,f,\chi\rt) &\ll Q^{-\hf} \max_{x\ll Q^{1+\vep}} \frac{S^\hf\log L}L \\
		&\ll \frac{\log L}{L}(QL + Q^{\hf}L^3 + Q^{\hf+\theta}L^{3+\vep})^\hf.
	\end{align*}
	The quantity inside the parentheses is minimized when $L = Q^{\frac 14-\frac{\theta}{2}}$, after such a choice we get the bound
	\begin{equation} \label{subconvexitybound}
		L\lt(\hf,f,\chi\rt) \ll Q^{\frac 38+\frac{\theta}{4}+\vep}.
	\end{equation}
	Thus Theorem \ref{thm:main} is proven.

\bibliographystyle{alpha}
\bibliography{MehmetBib}
\end{document}